\documentclass[11pt]{article}
\usepackage{amsfonts,amsmath,amssymb}
\usepackage{url}
\usepackage{epsfig,latexsym,graphicx}
\usepackage{esint}
\usepackage{graphicx,latexsym,amssymb,amsmath,amsfonts,fancyhdr,amsthm}
\usepackage{euscript}
\newtheorem{maintheorem}{Theorem}

\newtheorem{theorem}{Theorem}[section]
\newtheorem{lemma}[theorem]{Lemma}

\newtheorem{proposition}[theorem]{Proposition}
\newtheorem{observation}[theorem]{Observation}

\newtheorem{definition}[theorem]{Definition}

\def\XXint#1#2#3{{\setbox0=\hbox{$#1{#2#3}{\int}$ }
\vcenter{\hbox{$#2#3$ }}\kern-.6\wd0}}

\newcommand{\E}{\mathbb{E}}
\renewcommand{\P}{\mathbb{P}}

\newcommand{\Z}{\mathbb{Z}}

\newcommand{\cf}{\mathcal{F}}

\begin{document}
\title{Evolving Voter Model on Dense Random Graphs}

\author{Riddhipratim Basu\thanks{Dept. of Statistics, University of California, Berkeley. Supported by UC Berkeley Graduate Fellowship. Email:riddhipratim@stat.berkeley.edu}
\and
Allan Sly\thanks{University of California, Berkeley and Australian National University. Supported by an Alfred Sloan Fellowship and NSF grants DMS-1208338, DMS-1352013. Email:sly@stat.berkeley.edu}
}
\date{}
\maketitle
\begin{abstract}
In this paper we examine a variant of the voter model on a dynamically changing network where agents have the option of changing their friends rather than changing their opinions. We analyse, in the context of dense random graphs, two models considered in Durrett et. al.~\cite{Durrett12}.  When an edge with two agents holding different opinion is updated,  with probability $\frac{\beta}{n}$, one agent performs a voter model step and changes its opinion to copy the other, and with probability $1-\frac{\beta}{n}$, the edge between them is broken and reconnected to a new agent chosen randomly from (i) the whole network (rewire-to-random model) or, (ii) the agents having the same opinion (rewire-to-same model). We rigorously establish in both the models, the time for this dynamics to terminate exhibits a phase transition in the model parameter $\beta$. For $\beta$ sufficiently small, with high probability the network rapidly splits into two disconnected communities with opposing opinions, whereas for $\beta$ large enough the dynamics runs for longer and the density of opinion changes significantly before the process stops. In the rewire-to-random model, we show that a positive fraction of both opinions survive with high probability.
\end{abstract}

\section{Introduction}
\label{s:intro}
In  recent years, a significant research effort in various fields, including biology, ecology, economics, sociolgy among others, has been concentrated on studying and modelling behaviour of large complex networks with many interacting agents. Different dynamics on large networks has been studied focussing on the structural impact of these dynamics on different models of networks. Some of the problems which received attention are consensus of opinion and polarisation, spread of epidemics, information cascades etc. (see \cite{N10, Dur06}). In many real world networks the evolution of the links in the network depend upon the states of the connecting agents and vice versa. The general class of network models that model this dependence are called adaptive or coevolutionary networks (see \cite{GB08, SP00}). As in the case of static networks, the problems of spread of information and epidemic, evolution of opinion and polarization into communities have been studied numerically and also using a variety of rigorous and partly non-rigorous methods (\cite{Volz07,Volz09,Gil06, KB08, Henry11}, see also \cite{Durrett12,Vas13} and references therein for more background). The problems we consider in this paper belong to this general class.

The voter model has classically been studied in the probability literature as an interacting particle system mainly on lattices \cite{HL1975, Lig85}. More recently voter models have been studied in the context of general networks as a model for spread of opinion \cite{PRL05, PRE05}. In the classical voter model on a fixed graph, each vertex has one of the two prevalent opinions, neighbours interact at some fixed rate, and one of the neighbours adopts the opinion of the other after the interaction. A simplified model of coevolution of network and opinion was introduced and studied using non-rigorous methods by Holme and Newman \cite{HN06} where they try to model the property that an agent is less likely to interact (remain connected with) another agent if there opinions do not match. Their model is similar to the classical voter model (but with number of opinions proportional to the size of the network) but with the added feature that, whenever there is an interaction between two vertices (agents) with different opinion, with probability $\alpha\in (0,1)$, one of the vertices breaks the link and connects to a different vertex of the same opinion, i.e., the network connections evolve with time as well. Using finite size scaling, Holme and Newman conjectured a phase transition in $\alpha$, where in the supercritical phase all the opinions will have small number of followers, but in the subcritical phase a giant community holding the same opinion will emerge.
This model and its further extensions were investigated in \cite{PRE08KH,PRL08}.

Durrett et al. \cite{Durrett12}, studied two variants of this model  using numerical simulations and formulated some conjectures about the behaviour of the model. They use certain non-rigorous and numerical methods to formulate some conjectures about the asymptotic behaviour of the models on sparse random graphs as network size becomes large. They take the initial network to be a sparse Erd\H{o}s-R\'enyi graph $G$ with average degree $>1$   and the two initial opinions distributed as product measure. In each step a uniformly chosen disagreeing edge is selected and a voter model step is performed with probability $1-\alpha$ and a rewiring step with probability $\alpha$.  They consider two variants of reconnecting edge (i) rewire-to-random where the edge is connected to a randomly chosen neighbour and (ii) rewire-to-same where the edge is connected to a random neighbour of the same opinion.
Based on numerical evidence and heuristics Durrett et. al. conjecture in~\cite{Durrett12} that

\begin{enumerate}
\item[(1)] Supercritical phase: In both variants, there exist $\alpha_c(u)\in (0,1)$ independent of initial density of $1$'s such that for $\alpha >\alpha_c$, the process reaches an absorbing state in time $O(n)$ and the final fraction $\rho$ of minority opinion is $\approx u$.

\item[(2)] Subcritical phase: In the {\bf rewire-to-random} model, for $\alpha<\alpha_c(u)$, the time to absorption is order $n^2$ on average and at the absorption time, the density of the minority opinion $\rho$ is bounded away from $0$ and independent of $u$.  In  contrast, for the {\bf rewire-to-same} model, $\rho\approx 0$, so one of the opinions takes over almost the whole network at the time of absorption.
\end{enumerate}

We establish analogous results in the dense case but without establishing a sharp transition. We also prove that both opinions survive  in {\bf rewire-to-random} model, however we cannot prove the contrasting result for the {\bf rewire-to-same} model as is conjectured in \cite{Durrett12}. Durrett et al. also formulates conjectures about finer behaviours of the evolving voter model along the path to absorption (see Conjecture 1 and Conjecture 2 in \cite{Durrett12}). Further extensions of these models with different social dynamics and multiple possible opinions were considered in \cite{PRE13, RG13PRE, MM13, Dur13PRE}.

\subsection{Main Results}

In this paper, we study the dense version of the model from~\cite{Durrett12} where the initial graph is $G(n,1/2)$. It is easy to see that to obtain in non-trivial transition,  we must renormalize the opinion update rate to $1-\alpha= \beta/n$ (this is due to the average degree being linear in $n$).
Let $\tau$ denote the time to reach an absorbing state, i.e., $\tau$ is the first time when there are no disagreeing edges in the graph (for the rewire-to-same model absorbing states are slightly different, see \S~\ref{s:def} below). For $\frac{1}{2}>\varepsilon>0$, let $\tau_{*}(\varepsilon)$ be the first time that the fraction of the minority opinion reaches $\varepsilon$, i.e., $\tau_*(\varepsilon)=\min\{t:N_*(t)\leq \varepsilon n\}$, where $N_*(t)$ is the number of vertices holding the minority opinion at time $t$. Now we state our main theorems.

\begin{maintheorem}
\label{t:rewiretorandom}
Let $\frac{1}{2}>\varepsilon '>0$ be given. For both variants of the model, there exist $0<\beta_0<\beta_{*}(\varepsilon')<\infty$ such that each of the following hold.
\begin{enumerate}
\item[(i)] For all $\beta<\beta_0$ and any $\eta>0$, we have $\{\tau<10n^2, N_*(\tau)\geq \frac{1}{2}-\eta\}$ holds with high probability.

\item[(ii)] For all $\beta>\beta_*(\varepsilon')$, we have that $\tau_{*}(\varepsilon')\leq \tau$ with high probability and
\[
\lim_{c \downarrow 0} \liminf_n \P[\tau >cn^3] = 1.
\]
\end{enumerate}
\end{maintheorem}

The following theorem addresses the issue of fraction of minority opinion in the process at the absorption time for the rewire-to-random model.

\begin{maintheorem}
\label{t:rtoresplit}
Let $\beta>0$ be fixed. For the rewire-to-random model there exists $\varepsilon_{*}=\varepsilon_{*}(\beta)>0$ such that $\tau<\tau_*(\varepsilon_{*})$ with high probability.
\end{maintheorem}

\subsection{Formal Model Definitions}
\label{s:def}
Now we describe formally the models we consider in this paper. Let $n$ be a fixed positive integer. Let $V$ be a fixed set with $|V|=n$. Let $V^{(2)}$ denote the set of all unordered pairs in $V$, we shall call elements of $V^{(2)}$ as \emph{bonds}. Also let $\mathcal{E}$ be a fixed set. We consider discrete time Markov chains $\{G(t)\}_{t\geq 0}$ taking values in
$$\left\{\{0,1\}^{V}, \left(V^{(2)}\right)^{\mathcal{E}}\right\},$$
i.e., for each $t$, $G(t)$ is a multi-graph on the vertex set $V$ with labelled edges coming from the set $\mathcal{E}$ (each edge in $\mathcal{E}$ is placed at one of the bonds); each vertex has one of the two opinions $0$ and $1$. The following notations will be used throughout this paper.

The opinion of a vertex $v$ at time $t$ shall be denoted by $v(t)$. The vector of opinions of vertices in $G(t)$ shall be denoted by $V(t)$. We shall denote by $N_0(t)$ and $N_1(t)$ the number of $0$s and $1$s in $V(t)$ respectively. Let $N_*(t)=\min\{N_0(t),N_1(t)\}$. For $v\in V$, let $C_v(t)$ denote the set of all vertices in $V$ which have the same opinion as $v$ in $G(t)$. By $\tilde{G}(t)=(V,E(t))$, we denote the underlying graph of $G(t)$. Often, when there is no scope of confusion we shall use $G(t)$ instead of $\tilde{G}(t)$ to denote the same. Notice that we are allowing multi-edges but not self loops, i.e., at a time $t$, a bond $(u,v)\in V^{(2)}$ may be connected by more than one edge, but there are no edges connecting $v$ to itself. For an edge connecting the bond $(u,v)$ in $G(t)$ we shall call it {\bf disagreeing} if $u(t)\neq v(t)$ and {\bf agreeing} otherwise.

{\bf Initial condition:} To simplify matters we only consider the following initial condition. We take $\tilde{G}(0)$ is distributed as $G(n,\frac{1}{2})$, i.e., each bond contains $0$ edge with probability $\frac{1}{2}$ and $1$ edge with probability $\frac{1}{2}$ independent of each other. Also, let $\{v(0)\}_{v\in V}$ be $i.i.d.$ $\mbox{Ber}(\frac{1}{2})$. Also denote the set of the labelled edges $\mathcal{E}=\{e_0,e_1,\ldots, e_N\}$.

{\bf Transition Probabilities:} We describe the one step evolution of the two variants of the Markov chains as follows.

Let $G(t)$ be the state of the chain at time $t$. Let $Z(t)=\mbox{Ber}(\frac{\beta}{n})$ be independent of $G(t)$. If $Z(t)=1$ we obtain $G(t+1)$ from $G(t)$ by taking a \emph{voter model step}, and if $Z(t)=0$ then we obtain $G(t+1)$ from $G(t)$ by taking a \emph{rewiring step}. We call $\beta>0$  the \emph{relabelling rate}, this is a parameter of the model.

Let $\mathcal{E}^{\times}(t)\subseteq \mathcal{E}$ denote the set of edges that are disagreeing in $G(t)$. Choose one edge $e$ from $\mathcal{E}^{\times}(t)$ uniformly at random. Let $(u,v)$ be the bond which this edge connects in $G(t)$. Choose one vertex randomly among $u$ and $v$, say $u$. The vertex $u$ as above will be called the {\bf root} of a rewiring move. 

{\bf The voter model step (relabelling step):} If $Z(t)=1$, then $u$ adopts the opinion of $v$, i.e., $G(t+1)$ is obtained from $G(t)$ by taking $\tilde{G}(t+1)=\tilde{G(t)}$, $v'(t+1)=v'(t)$ for all $v'\in V\setminus \{u\}$ and $u(t+1)=v(t+1)$.

{\bf The rewiring step:} If $Z(t)=0$, the two chains we consider evolve differently.

{\bf Rewire-to-random model:} In this model, we choose a vertex $v'$ uniformly at random from $V \setminus \{u\}$. We obtain $G(t+1)$ from $G(t)$ by taking $V(t+1)=V(t)$ and $E(t+1)$ is obtained from $E(t)$ by removing the edge $e$ from the bond $(u,v)$ and adding it to the bond $(u,v')$.

{\bf Rewire-to-same model:}
In this model, we choose a vertex $v'$ uniformly at random from $C_u(t)\setminus \{u\}$. We obtain $G(t+1)$ from $G(t)$ by taking $V(t+1)=V(t)$ and $E(t+1)$ is obtained from $E(t)$ by removing the edge $e$ from the bond $(u,v)$ and adding it to the bond $(u,v')$.

We make the following basic observation characterising the absorbing states.
\begin{observation}
\label{o:absorbing}
Notice that, on finite networks both the chains are absorbing. For the {\it rewire-to-random} model the only absorbing states are those which corresponds to the graph having no disagreeing edges, i.e., either one opinion has taken over all the vertices, or the graph is split into disconnected communities, where all the vertices in a community has the same opinion. For the {\it rewire-to-same} model the absorbing states are those that either have no disagreeing edges, or those in which one of the opinions are held by only one vertex.
\end{observation}
%
Notice that the number of edges is conserved in each step of the chain, i.e., we have that $|E(t)|=|\mathcal{E}|$ for all $t$. Also observe that even though we have labelled edges, this fact does not affect the behaviour of the model at all. The edges are labelled, simply because it will be convenient while constructing a coupling of this chain with another process which we shall use.

One of the main questions we are interested in for both the models described above is the asymptotics of absorption time as a function of $\beta$ as $n\rightarrow \infty$, and whether it exhibits a phase transition in $\beta$ or not. Notice that if $\beta=0$, then we have only rewiring moves the absorption time is $\Theta(n^2)$, i.e., the graph splits immediately into two communities having different opinions. We investigate whether similar phenomenon occurs if $\beta>0$ is sufficiently small. In the other extreme, if the rewiring moves are much rarer compared to the relabelling moves (i.e., $\beta>>1$) one might expect the model to behave similarly as the voter model on a static graph, where the minority opinion density will become very small before reaching an absorbing state, and the absorption time will be at least $\Theta(n^3)$. This is established in Theorem \ref{t:rewiretorandom}. A related quantity of interest is the fraction of the minority opinion vertices when the process reaches an absorbing state. For $\beta$ sufficiently large does the minority opinion persist with a positive fraction? Theorem \ref{t:rtoresplit} provide the answers for the rewire-to-random model.

\subsection{Outline of the proof}
We prove parts $(i)$ and $(ii)$ of Theorem \ref{t:rewiretorandom} separately. The arguments are similar for the rewire-to-random model and the rewire-to-same model. We provide details only for the rewire-to-random model while pointing out the differences for the rewire-to-same model. To prove part $(i)$, we essentially show that before the density of either opinion changes, a rewiring move is likely to decrease the number of disagreeing edges. Using a martingale argument we show that, by time $\Theta(n^2)$ (by which time the opinion densities cannot change significantly), the number of disagreeing edges decay to 0.

Most of the work goes into proving part $(ii)$ of Theorem \ref{t:rewiretorandom}. We show that for $\beta=\beta(\varepsilon')$ sufficiently large, the graph $G(t)$ remains close enough to an Erd\H{o}s-R\'enyi graph, in a sense to be made precise, as long as the minority opinion density does not drop below $\varepsilon'$. To this end, we define a number of stopping times detecting when $G(t)$ deviates too much from an Erd\H{o}s-R\'enyi graph for the first time with respect to certain different properties, and roughly show that all those stopping times are with high probability at least as large as than $\tau_{*}(\varepsilon')$. The properties we need to consider are vertex degrees, the Cheeger constant and edge-multiplicities.

Corresponding to each of the properties we consider, we define two stopping times, one with a stronger threshold and the other with a weaker threshold. We show that provided none of the weaker thresholds have been reached, the opinions quickly mix to an approximate product measure which guarantees that the properties of interest are sufficiently mean reverting for our purposes.

For the proof of Theorem \ref{t:rtoresplit}, we show that for a fixed $\beta$, there exist sufficiently small but positive $\varepsilon_{*}$, such that once the minority opinion reaches $\varepsilon_{*}$, the typical vertices having minority opinions start losing disagreeing edges at a higher rate than it gains them, and eventually all the disagreeing edges are lost before the minority opinion density can change substantially.

{\bf Organisation of the paper:} The rest of the paper is organised as follows. In \S~\ref{s:smallbeta}, we prove part $(i)$ of Theorem \ref{t:rewiretorandom} for the rewire-to-random model. Most of the work in this paper goes towards the proof of part $(ii)$ of Theorem \ref{t:rewiretorandom} for the rewire-to-random model, which spans \S~\ref{s:stoptime}, \S~\ref{s:weakbound} and \S~\ref{s:strongbound}. In \S~\ref{s:stoptime} we define all the stopping times that we need to use. In \S~\ref{s:weakbound} we show that, if by time $t$ the graph does not reach any of the stronger stopping times, then the graph does not reach any of the weaker stopping times by time $t+t'$ with high probability, where $t'$ is of order $n^2$. That the the graph is also unlikely to reach any of the strong stopping times by time $t+t'$ as long as the minority opinion density does not become too small is shown in \S~\ref{s:strongbound}. Together these complete the proof of Theorem \ref{t:rewiretorandom}, part $(ii)$. Theorem \ref{t:rtoresplit} is proved in \S~\ref{s:esplit}. In \S~\ref{s:samemod}, we point out the significant adaptations to the argument that are necessary to prove Theorem \ref{t:rewiretorandom} for the rewire-to-same model. We finish with the discussion of some open problems in \S~\ref{s:conclusion}.

\section{Fast polarization for small $\beta$}
\label{s:smallbeta}
In this section we prove part $(i)$ of Theorem \ref{t:rewiretorandom} for the \emph{rewire-to-random model} with relabelling rate $\beta$. First we make the following definitions. Let $D_{\max}(t)$ denote the maximum degree of a vertex in $G(t)$. The degree of a vertex is defined as the number of edges incident to it, and not the number of bonds containing edges. Also, let $X_t=|\mathcal{E}^{\times}(t)|$. Consider the following stopping times.   Let $\tau_1=\min\{t: D_{\max}(t)\geq 8n\}$ and let $\tau_{2}=\tau_*(\frac{1}{3})$, i.e., $\tau_{2}=\min\{t: N_{*}(t)=\min\{N_0(t),N_1(t)\}\leq \frac{n}{3}\}$. Define $\tau_0=\tau \wedge \tau_1 \wedge \tau_2$. We have the following lemmas.

\begin{lemma}
\label{l:betasmallmg}
There exists $\beta_0>0$, such that for all $\beta<\beta_0$, we have $\tau_0\leq 6n^2$ with high probability.
\end{lemma}

\begin{proof}
Let  $\cf_{t}$ denote the filtration generated by the process up to time $t$. Observe that whenever an edge is rewired, $X_t$ either remains the same or decreases by $1$. Conditional on $\cf_t$, the chance that $X_t$ is decreased by a rewiring move is at least $\frac{N_{*}(t)-1}{n-1}$. Also notice that a relabelling move, i.e., a voter model step can increase $X_t$ by at most $D_{\max}(t)$. Hence  we have for $\lambda>0$,
\begin{equation}
\label{e:supmart}
\E\left(e^{\frac{\lambda X_{t+1}}{n}} \mid \cf_t\right) \leq  e^{\frac{\lambda X_{t}}{n}}\left((1-\frac{\beta}{n})\left(1+\frac{N_{*}(t)-1}{n-1}(e^{-\frac{\lambda}{n}}-1)\right)+\frac{\beta}{n} e^{\frac{\lambda D_{\max}(t)}{n}}\right).
\end{equation}
Now for large $n$ on the event that $\{t<\tau_0\}$ we have $\frac{N_*(t)-1}{n-1}\geq \frac{1}{4}$. Taking $\lambda>0$ sufficiently small so that $e^{8\lambda}\leq 1+9\lambda$ and $e^{-\frac{\lambda}{n}}-1\leq -\frac{\lambda}{2n}$. Then  on $\{t< \tau_0\}$, we have for $\beta<\frac{1}{400}$,
\begin{eqnarray}
\label{e:supmart2}
\E\left(e^{\frac{\lambda X_{t+1}}{n}} \mid \cf_t\right) &\leq
&  e^{\frac{\lambda X_{t}}{n}}\left((1-\frac{\beta}{n})\left(1-\frac{1}{4}(e^{-\frac{\lambda}{n}}-1)\right)+\frac{\beta}{n} e^{8\lambda}\right)\nonumber\\
&\leq & e^{\frac{\lambda X_{t}}{n}}\left((1-\frac{\beta}{n})(1-\frac{\lambda}{8n}) +\frac{\beta}{n}(1+9\lambda)\right)\nonumber\\
&\leq & e^{\frac{\lambda X_{t}}{n}}(1-\frac{\lambda}{10n}).
\end{eqnarray}
It follows from above that
\begin{equation}
\label{e:markov}
\P[\tau_0 > t\mid \cf_0]\leq \E\left[e^{\frac{\lambda X_t}{n}}1_{\{\tau_0 >t\}}\mid \cf_{0}\right] \leq
e^{\frac{\lambda X_0}{n}}e^{-\frac{\lambda t}{10n}}\leq e^{\frac{\lambda n}{2}}e^{-\frac{\lambda t}{10 n}}
\end{equation}
since $X_0\leq \frac{n^2}{2}$. Hence we have
\begin{equation*}
\label{e:tau}
\P[\tau_0 > 6n^2]\leq e^{-\frac{\lambda n}{10}}.
\end{equation*}
\end{proof}

\begin{lemma}
\label{l:betasmallincoming}
We have $\tau_1>6n^2$ with high probability.
\end{lemma}

\begin{proof}
It is easy to see that the \emph{rewire-to-random} dynamics can be implemented in the following way. Without loss of generality let $V=[n]$. And let $\mathbb{W}=\{W_{i}\}_{i\geq 1}$ be a sequence of i.i.d. random variables with each $W_i$ being uniformly distributed over $\{1,2,\ldots, n\}$. Let us define $L_0=0$, and we define $L_i$ recursively as follows. Let, $v_i$ be the \emph{root} of the $i$-th rewiring move. Then we define $L_{i}=\min\{j>L_{i-1}:W_{j}\neq v_i\}$. Then in the $i$-th rewiring move, we add an edge to the bond $(v_i,W_{L_i})$. The algorithm can be described as follows. For each rewiring move, start inspecting the list $\mathbb{W}$ from the first previously uninspected element upto the first time you find a vertex which is not equal to the root of the current rewiring. Rewire the edge to this vertex. Clearly, in this way the chosen vertex is uniform among all vertices other than $v_i$, and hence this is indeed an implementation of the \emph{rewire-to-random} dynamics.

Now observe that $L_{i+1}-L_{i}$ are i.i.d. $\mbox{Geom}(\frac{n-1}{n})$ variables. It follows by a large deviation estimate that $L_{6n^2}<\frac{13n^2}{2}$ with exponentially high probability. Now for $v\in V$, let
$$N(v)=\#\{i\leq \frac{13n^2}{2}: W_{i}=v\}.$$
Clearly, $N(v)$ is distributed as $\mbox{Bin}(\frac{13n^2}{2},\frac{1}{n})$, and a Chernoff bound implies
$$\P[N(v)\geq 7n] \leq e^{-n/78}.$$
Hence, noting that $D_{\max}(0)\leq n$, we have using a union bound over all the vertices

\begin{equation}
\label{e:tau1}
\P[\tau_1 \leq 6n^2]\leq ne^{-n/78}+\P[L_{6n^2}>\frac{13n^2}{2}].
\end{equation}
This completes the proof of the lemma.
\end{proof}

\begin{lemma}
\label{l:betasmalldensity}
There exists $\beta_0>0$, such that for all $\beta<\beta_0$, we have $\tau_2\geq 6n^2\wedge \tau$ with high probability.
\end{lemma}

\begin{proof} For $t\geq 1$, it is easy to see that $RL(t)$, the number of relabelling moves upto time $t$, is stochastically dominated by a $\mbox{Bin}(6n^2,\frac{\beta}{n})$ variable. On $\{t<\tau\}$, we have that $N_0(t)-N_0(0)$ is distributed as $Z_{RL(t)}$ where $\{Z_i\}_{i\geq 0}$ is a simple symmetric random walk on $\Z$ started from $0$. Using a union bound it follows that

$$\P[\tau _2 < 6n^2\wedge \tau]\leq \P[N_{*}(0)\leq \frac{2n}{5}]+ \P[RL(6n^2)>12\beta n]+\P[\max_{i\leq 12\beta n}|Z_i| \geq \frac{3n}{20}].$$
By choosing $\beta_0$ sufficiently small, the last term in the above inequality is $0$ for all $\beta<\beta_0$. Noticing that $\P[N_{*}(0)<\frac{2n}{5}]=2\P[\mbox{Bin}(n,\frac{1}{2})<\frac{2n}{5}]$ and using Hoeffding inequality to bound the first term and a Chernoff bound on the second term yields

\begin{equation}
\label{e:tau2}
\P[\tau _2 < 6n^2\wedge \tau]\leq 2e^{-2n/25} +  e^{-2\beta n}.
\end{equation}
This completes the proof of the lemma.
\end{proof}

Now we are ready to prove Theorem \ref{t:rewiretorandom}$(i)$.

\begin{proof}[Proof of Theorem \ref{t:rewiretorandom}$(i)$]
From Lemma \ref{l:betasmallmg}, Lemma \ref{l:betasmallincoming} and Lemma \ref{l:betasmalldensity} it follows that with high probability we have  $\{\tau_0\leq 6n^2, \tau_1\wedge \tau_2\geq 6n^2\wedge \tau\}$. It follows that, $\tau\leq 6n^2$ with high probability. The second part of the theorem follows from noting that using a random walk estimate as in Lemma \ref{l:betasmalldensity}, we see that for each $\eta>0$, the probability that the density of the minority opinion drops below $\frac{1}{2}-\eta$, within $6n^2$ steps tends to $0$ as $n\rightarrow\infty$. This completes the proof of the theorem.
\end{proof}

\section{High relabelling rate case: Stopping times}\label{s:stoptime}
\subsection{A time change: \emph{Rewire-to-random-*} dynamics}
For the proof of Theorem \ref{t:rewiretorandom}$(ii)$, we shall consider a time changed variant of \emph{rewire-to-random} dynamics, which we call \emph{rewire-to-random-*} model. This model is same as the \emph{rewire-to-random} model, except that now at time $(t+1)$, instead of choosing a disagreeing edge at random, we choose an edge at random from $G(t)$. If the edge is not disagreeing, then we do nothing. It is clear that \emph{rewite-to-random-*} model is a slowed down version of \emph{rewire-to-random} model. It is also clear that if we prove Theorem \ref{t:rewiretorandom}$(ii)$ for the \emph{rewire-to-random-*} dynamics, then it will imply the same theorem for the \emph{rewire-to-random} dynamics.

{\bf Assumption on the initial condition:}
For this section and the next two, we shall always assume that $G(0)$ satisfies the following conditions.

\begin{enumerate}
\item[(i)] $|E(0)|$, the number of edges in $G(0)$ is in $[\frac{n^2}{4}-n^{3/2}, \frac{n^2}{4}+n^{3/2}]$.
\item[(ii)] $\#\{v\in V: v(0)=0\} \in [\frac{n}{2}-n^{3/4}, \frac{n}{2}+n^{3/4}]$.
\end{enumerate}
Since both the events hold with probability $1-o(1)$, this assumption does not affect any of our results.

Now we move towards proving Theorem \ref{t:rewiretorandom}$(ii)$. Let us fix $\varepsilon'>\varepsilon>0$ for the rest of this paper. Let $\tau_{*}=\tau_{*}(\varepsilon)$, i.e.,
$$\tau_*=\min\{t:\min(N_{0}(t),N_{1}(t))< \varepsilon n\}.$$

\textbf{Parameters:}
Now we define the following stopping times. In the definition of these stopping times and the proofs that follow we use a number of parameters that need to satisfy the following relationships. For a fixed $
\varepsilon$ our parameters satisfy the following inequalities.
$\varepsilon_2<\varepsilon^2/1000$, $\varepsilon_3< \varepsilon_2^{2}/1000$. $\varepsilon_7<\varepsilon_3^{2}/1000$. Fixing these parameters we choose $C_2=2$, $\varepsilon_4< \frac{1}{4\log 10}$, $\delta < \frac{\varepsilon_3}{10000C_2}$. We choose $\varepsilon_{14}< \varepsilon^2/100$. Fixing all these $C_1$ is chosen sufficiently large depending on $\delta$ and $\varepsilon$, here the exact functional dependence is not of interest to us. After fixing all these parameters, $C$ is chosen sufficiently large depending on these. There are also many other parameters used in the proofs which are chosen either sufficiently small or large depending on other parameters, again where the functional dependence is not of importance to us. Finally $\beta$ is taken sufficiently large depending on all the parameters used. Also we shall always take $n$ sufficiently large depending on everything else.

$\bullet$~{\bf Stopping times for large Cuts}:

Let $S$ and $T$ be two disjoints subsets of $V$ with $S\cup T=V$. We denote by $N_{ST}(t)$  the number of edges in $G(t)$ with one endpoint in $S$ and another endpoint in $T$. Define $N_{SS}(t)$ similarly. Also let $N(t)=N$ denote the total number of edges in $G(t)$. Let
\[ 
K_{ST}(t)=(\dfrac{N_{SS}(t)-\frac{1}{4}|S|^2}{N(t)})^2+ (\dfrac{N_{TT}(t)-\frac{1}{4}|T|^2}{N(t)})^2.
\]
Let $$L(t)=\max_{S,T:\min(|S|,|T|)\geq \varepsilon_2 n}K_{ST}(t).$$
Also let $$L'(t)=\max_{S,T:\min(|S|,|T|)\geq \varepsilon_2 n} \mid \dfrac{N_{ST}(t)-\frac{1}{2}|S||T|}{N(t)} \mid \vee \mid \dfrac{N_{SS}(t)-\frac{1}{4}|S|^2}{N(t)} \mid.$$

Now the two stopping times are defined as follows:
\begin{itemize}
\item The stronger stopping time: $\tau_2=\min\{t: L(t)\geq \varepsilon_3^2\}$.
\item The weaker stopping time: $\tau'_2=\min\{t: L'(t)\geq 2\varepsilon_3\}$.
\end{itemize}

$\bullet$~{\bf Stopping times for individual edge multiplicities}:

For $u,v\in V$, let $M_{uv}(t)$ denote the number of edges in the bond $(u,v)$ in $G(t)$. Let $M(t)=\max_{u\neq v}M_{uv}(t)$.
Now the two stopping times are defined as follows:
\begin{itemize}
\item The stronger stopping time: $\tau_3=\min\{t: M(t)\geq \varepsilon_4\log n\}$.
\item The weaker stopping time: $\tau'_3=\min\{t: M(t)\geq 2\varepsilon_4\log n\}$.
\end{itemize}

$\bullet$~{\bf Stopping times for balanced vertices}:

Let us call a vertex $v$ $\epsilon$-balanced if for all $k$, $\#\{u\in V: M_{uv}=k\}\leq \epsilon 10^{-k}n$. We define the two stopping times as follows.
\begin{itemize}
\item The stronger stopping time: $\tau_4=\min\{t:\exists v\in V ~\text{not}~C_1\text{-balanced in}~G(t)\}$.
\item The weaker stopping time: $\tau'_4=\min\{t:\exists v\in V ~\text{not}~2C_1\text{-balanced in}~G(t)\}$.
\end{itemize}

$\bullet$~{\bf Stopping times for maximum and minimum degrees}:

Let $D_{\max}(t)$ and $D_{\min}(t)$ denote the maximum and minimum degree in $G(t)$ respectively. The stopping times are defined as follows.

\begin{itemize}
\item The stronger stopping time: $\tau_5=\min\{t: D_{\max}(G(t))> (1-\frac{\varepsilon}{2})n~\text{or}~D_{\min}(t)<\frac{\varepsilon n}{2} \}$.
\item The weaker stopping time: $\tau'_5=\min\{t: D_{\max}(t)>C_2n~\text{or}~D_{\min}(G(t))<\frac{\varepsilon n}{4}\}$.
\end{itemize}

It is easy to see that for each $i=2,3,4,5$, we have $\tau'_i\leq \tau_i$.\\

Finally we define $\tau_0=\tau_{*}\wedge \tau_2\wedge \tau_3\wedge \tau_4\wedge \tau_5$ and $\tau'_0=\tau_{*}\wedge \tau'_2\wedge \tau'_3\wedge \tau'_4\wedge \tau'_5$.\\

Part $(ii)$ of Theorem \ref{t:rewiretorandom} follows from the next theorem.

\begin{theorem}
\label{t:stoptime}
There exist $\beta_{*}=\beta_{*}(\varepsilon)$ such that for all $\beta>\beta_{*}$, we have for the rewire-to-random-* model $\tau_0\geq \tau_*-n^2$ w.h.p.
\end{theorem}

We shall prove Theorem \ref{t:stoptime} over the next two sections. Before that we show how this implies part $(ii)$ of Theorem \ref{t:rewiretorandom} for \emph{rewire-to-random-*} model.

\begin{proof}[Proof of Theorem \ref{t:rewiretorandom},$(ii)$]
Notice that, it follows from a random walk estimate that $\tau_{*}\geq \tau_*(\varepsilon')+n^2$ with high probability. On $\{\tau_0\geq \tau_{*}(\varepsilon')\}$, we have that $\tau_{*}(\varepsilon')-1<\tau_0$. And hence, in particular, $\tau_{*}(\varepsilon')-1<\tau_2$. Let $S$ be the set of all vertices with the minority opinion at time $\tau_{*}-1$. Since $\varepsilon_2< \varepsilon<\varepsilon'$, we have that $N_{ST}(\tau_{*}-1)\geq \frac{1}{2}|S||T|-2\varepsilon_{3}N(t)>0$, since $2\varepsilon_{3}<\varepsilon_2(1-\varepsilon_{2})$. It then follows that $\tau\geq \tau_{*}(\varepsilon')$ for the \emph{rewire-to-random-*} dynamics. since the \emph{rewire-to-random-*} dynamics is merely a time changed version of the \emph{rewire-to-random} dynamics, the first statement in Theorem \ref{t:rewiretorandom},(ii) follows. The second statement is an obvious consequence after observing that on $\{t<\tau\}$ the number of vertices of one opinion does a simple random walk which takes one step with roughly $\frac{n}{\beta}$ steps of the rewire-to-random dynamics.   
\end{proof}

Before starting with the proof of Theorem \ref{t:stoptime} we prove the following lemma which establishes the connection between the evolving voter model dynamics and our stopping times that we shall exploit extensively.

\begin{lemma}
\label{l:mixing}
Consider the following continuous time random walk on $G(t)$. Let each directed edge ring at rate $\frac{\beta}{2n}$. Whenever an edge rings a walker at the starting point of the edge moves along the edge. Let $\lambda(G(t))$ denote the spectral gap of this Markov chain. Then there exists $\varepsilon_{14}>0$, such that we have, on $\{t<\tau'_2\wedge \tau'_3\wedge \tau'_4\wedge \tau'_5\}$, $\lambda(G(t))\geq \beta\varepsilon_{14}$.
\end{lemma}

\begin{proof}
Let $h(G(t))$ denote the Cheeger constant of the corresponding random walk. We have
$$h(G(t)):=\min_{S,T S\cup T=V, S\cap T=\emptyset, |S|\leq n/2} \frac{N_{ST}(t)\beta}{2|S|n}.$$
Now on $\{t\leq \tau'_2\}$, if $|S|\geq \varepsilon_2 n$, $N_{ST}(t)\geq \frac{1}{2}|S||T|-2\varepsilon_3 N(t)\geq \frac{1}{2}|S||T|-\varepsilon_3 n^2$ and hence
$$\frac{N_{ST}(t)\beta}{2|S|n}\geq \beta (\frac{1}{8}-\frac{\varepsilon_3}{2\varepsilon_2})\geq 2\beta \sqrt{C_2} \sqrt{\varepsilon_{14}}$$
provided
$$\varepsilon_{14}\leq \dfrac{(\frac{1}{8}-\frac{\varepsilon_3}{2\varepsilon_2})^2}{4C_2}.$$
On $\{t\leq \tau'_4 \wedge \tau'_5\}$, if $|S|\leq \varepsilon_2 n$, then
$$N_{SS}(t)+N_{ST}(t)\geq \frac{\varepsilon|S|n}{4}$$
and
$$N_{SS}(t)\leq 2|S|n\varepsilon_2\log \frac{1}{\varepsilon_2}.$$

It follows that

$$\frac{N_{ST}(t)\beta}{2|S|n}\geq \beta (\frac{\varepsilon}{8}-2\varepsilon_2\log \frac{1}{\varepsilon_2})\geq 2\beta \sqrt{C_2} \sqrt{\varepsilon_{14}}$$
provided
$$\varepsilon_{14}\leq  \dfrac{(\frac{\varepsilon}{8}-2\varepsilon_2\log \frac{1}{\varepsilon_2})^2}{4C_2}.$$

Now it follows that on $\{t\leq \tau'_2\wedge \tau'_4\wedge \tau'_5\}$,

$$h(G(t))\geq 2\beta \sqrt{C_2} \sqrt{\varepsilon_{14}}.$$

Now, using Cheeger inequality (Theorem 13.14 of \cite{LPW09}, see \cite{FN02} for the variant used here) we get,

$$\lambda(G(t))\geq \frac{h(G(t))^2}{2\frac{\beta D_{max}(t)}{n}}\geq \frac{h(G(t))^2}{2\beta C_2}\geq \beta \varepsilon_{14}$$
which completes the proof of the lemma.
\end{proof}


%

\section{Estimates for the weak stopping times}
\label{s:weakbound}
In this section we show that provided the process has not reached any of the stronger stopping times by time $t$, i.e., $t<\tau$, then within a small number ($\delta n^2$) steps, the process is unlikely to reach any of the weaker stopping times, i.e.,  $t+\delta n^2<\tau'_0$ with high probability. The general idea is that by time $\delta n^2$, there are not enough rewiring steps to change the graph substantially. We start with the following lemma which controls the fraction of minority opinion vertices in the time interval $\{t+1,t+2,\ldots, t+\delta n^2\}$.



\begin{lemma}
\label{lbw0}
We have that $\P(\tau_{*}(\frac{4\varepsilon}{5})\leq t+\delta n^2\mid \cf_{t},t<\tau_0)\leq e^{-cn}$ where $c=c(\delta,\beta,\varepsilon)>0$.
\end{lemma}
\begin{proof}
It follows from a Chernoff bound that the probability that there are more than $2\delta\beta n$ many relabelling steps in $[t+1,t+\delta n^2]$ is exponentially small in $n$. Note that the number of vertices of a certain opinion does a simple symmetric random walk in absorbed at $0$ or $n$ in the \emph{rewire-to-random} dynamics. The lemma now follows from a random walk estimate by observing that \emph{rewire-to-random-*} dynamics is slower than rewire-to-random dynamics.
\end{proof}


The next lemma considers the weaker stopping times for the large cuts.  

\begin{lemma}
\label{l:lbw2}
On the event $\{t<\tau_0\}$, $t+\delta n^2<\tau'_2$.
\end{lemma}
\begin{proof}
Clearly for all $S,T$ that make a partition of $V$ and any $t'$ with $t\leq t'\leq t+\delta n^2$,
$|N_{ST}(t')-N_{ST}(t)|\vee |N_{SS}(t')-N_{SS}(t')|\leq \delta n^2$. It follows from definitions that if $\delta <\frac{\varepsilon_3}{100}$ and $1000\varepsilon_7<\varepsilon$, then $t+\delta n^2\leq \tau'_2$.
\end{proof}

The next lemma shows the weaker degree estimates continue to hold till time $t+\delta n^2$  with high probability if $t<\tau_0$.

\begin{lemma}
\label{l:lbw5}
We have $\P(t+\delta n^2\geq \tau'_5 \mid \cf_t, t<\tau_0)\leq e^{-cn}$ for some constant $c>0$.
\end{lemma}

\begin{proof}
Condition on $\cf_{t}$. For any fixed vertex $v$, the number of times in $[t+1,t+\delta n^2]$ an edge is rewired to $v$ is stochastically dominated by a $\mbox{Bin}(\delta n^2, \frac{1}{n-1})$ variable. By  Chernoff's inequality and a union bound it follows that the probability that any vertex gets more than $2\delta n$ edges is exponentially small in $n$. It follows that with exponentially high probability $\max_{t'\in [t+1,t+\delta n^2]}D_{\max}(t')<C_2n$ provided $C_2> (1-\frac{\varepsilon}{2})+2\delta$.


For the lower bound notice that the probability that a vertex $v$ with degree at least $\varepsilon n/2$ at time $t$ becomes of degree less that $\varepsilon n/4$ in time $[t+1,t+\delta n^2]$ is at most
$$\P\left(\mbox{Bin}\biggl(\delta n^2, \frac{5\varepsilon}{4n}\biggr) \geq \varepsilon n/4\right).$$
The above probability is exponentially small in $n$ by a Chernoff bound provided $6\delta < 1$. Taking a union bound over all the vertices completes the proof of the lemma.

%
%
\end{proof}

Now we prove a similar statement for individual edge-multiplicities.

\begin{lemma}
\label{l:lbw3}
We have $\P(t+\delta n^2 \geq  \tau'_3\mid \cf_{t}, t<\tau_0) \leq \frac{1}{n^{20}}$.
\end{lemma}

\begin{proof}
For every bond  $(u,v)$ in $V^{(2)}$, it follows from Lemma \ref{l:lbw5} that with exponentially high probability  $D_u(t')+D_v(t')\leq 2C_2 n$ for all $t'\in [t+1,t+\delta n^2]$ where $D_w(t')$ denotes the degree of the vertex $w$ in $G(t')$. Let $A_{uv}$ denote that event.
Then, on $A_{uv}$, the number of edges added to the bond $(u,v)$ is stochastically dominated by
$$\mbox{Bin}(\delta n^2, \frac{10C_2}{n(n-1)})\preceq \mbox{Bin}(\delta n^2, \frac{12C_2}{n^2})$$
variables provided $\varepsilon>16\varepsilon_3$ and $n$ sufficiently large. By a Chernoff bound again, provided $n$ is sufficiently large we get that
$$\P\left( \mbox{Bin}(\delta n^2, \frac{12C_2}{n^2})\geq \varepsilon_4\log n \right)\leq \frac{1}{n^{23}}.$$
Taking a union bound over all bonds completes the proof of the lemma.
\end{proof}

Finally we work out the estimates for number of multi-edges incident to a given vertex.

\begin{lemma}
\label{l:lbw4}
We have $\P(t+\delta n^2\geq  \tau'_4\mid \cf_t, t< \tau_0)\leq \frac{1}{n^{18}}$.
\end{lemma}
\begin{proof}
Condition on $\cf_t$. As in the previous lemma we shall ignore without loss of generality the event with exponentially small probability that for some $t'\in [t+1,t+\delta n^2]$ there exists $u,v\in V$ such that $D_u(t')+D_v(t')>2C_2n$. Fix $v\in V$ and $0<k<2\varepsilon_4\log n$. Let  $N(\ell\rightarrow k)$ denote the number of bonds containing $v$ that gained at least $(k-\ell)$ edges during time $[t+1,t+\delta n^2]$. Clearly,
$$\max_{t'\in [t+1,t+\delta n^2]}\#\{u:M_{uv}(t')=k\}\leq C_1 10^{-k}n+\sum_{\ell=0}^{k-1} N(\ell \rightarrow k)+ \#\{u:M_{uv}(t)>k\}.$$

Clearly the third term in the above sum is bounded by $\frac{1}{9}C_1 10^{-k}n$. To bound the second term fix $\ell$ with $0\leq \ell \leq k-1$. Let $u_1,u_2,\ldots, u_D$ be the set of vertices in $V$ such that $M_{vu_i}(t)=\ell$. Without loss of generality we can assume $D=C_110^{-\ell}n$. Let $T_i$ denote the number of edges gained by the bond $(v,u_i)$ in $[t+1,t+\delta n^2]$. We construct a family of random variables $(Y_1,Y_2,\ldots ,Y_D)$ which jointly stochastically dominates $(T_1,T_2,\ldots, T_D)$ as follows. 
 
Consider $D$ urns all of which are empty to begin with. For $i=\{1,2,\ldots, \delta n^2\}$, at the ith step with probability $\frac{12C_2D}{n^2}$ we choose one of urns uniformly and put a ball in it. Let $(Y_1,Y_2,\ldots, Y_D)$ denote the vector of the number of balls in the urns after $\delta n^2$ steps. Notice that since at each step in the  voter model dynamics only one of the bonds $(v,u_i)$ can gain an edge, and since for each of the bonds, the chance to gain an edge is at most $\frac{12C_2}{n^2}$ as shown in the previous lemma, it follows that $(Y_1,Y_2,\ldots, Y_D)$ indeed jointly  stochastically dominates $(T_1,T_2,\ldots, T_D)$. Observe further that conditional on $\sum_{i=1}^{D}Y_i\leq 24\delta C_2 D$, the joint law of $(Y_1,Y_2,\ldots, Y_D)$ is stochastically dominated by the conditional joint law of $(Y'_1,Y'_2,\ldots, Y'_D)$ conditional on $\sum_{i=1}^{D}Y'_i\geq 24\delta C_2 D$ where $\{Y'_i\}$ is an i.i.d. sequence of $\mbox{Bin}(24\delta C_2 D, \frac{2}{D})$ variables. Let $Z_i$ (resp. $Z'_i$) denote the indicator of $Y_i\geq (k-\ell)$ (resp. $Y'_i\geq (k-\ell)$). Let $\mathcal{C}$ (resp. $\mathcal{C}'$) denote the event that $\sum_{i=1}^{D}Y_i\leq 24\delta C_2 D$ (resp. $\sum_{i=1}^{D}Y'_i\geq 24\delta C_2 D$). It follows therefore that 
\begin{eqnarray*}
\P[\sum_{i=1}^D Z_i\geq 25^{-(k-\ell)}D]& \leq & \P[\mathcal{C}^c]+\P[\sum_{i=1}^D Z_i\geq 25^{-(k-\ell)}D\mid \mathcal{C}\\
&\leq & e^{-4\delta C_2 D}+\P[\sum_{i=1}^D Z'_i\geq 25^{-(k-\ell)}D\mid \mathcal{C}']\\
&\leq & e^{-4\delta C_2 D}+ 2\times \P[\mbox{Bin}(D,q(k-\ell))\geq 25^{-(k-\ell)}D]
\end{eqnarray*} 
where $q(k-\ell)=\mbox{Bin}(24\delta C_2 D, \frac{2}{D})\geq (k-\ell))$. In the above equation we have used Chernoff bounds to deduce $\P[\mathcal{C}^c]\leq e^{-4\delta C_2 D}$ and $\P[\mathcal{C}_1]
\geq \frac{1}{2}$. Using another Chernoff bound we get
$$q(k-\ell)\leq e^{-\frac{1}{2}(k-\ell)\log (\frac{1}{48C_2\delta})(k-\ell)}\leq 30^{-(k-\ell)}$$
provided $1>48000C_2\delta$.
%

Since $N(\ell\rightarrow k) \preceq \sum Z_i$ it  follows using Chernoff's inequality yet again that

$$\P(N(\ell\rightarrow k)\geq (2.5)^{-(k-\ell)}C_1 10^{-k}n)\leq e^{-\frac{C_1 10^{-k}n}{75}}\leq 2e^{-\frac{C_1\sqrt{n}}{75}}+e^{-4\delta C_1C_2\sqrt{n}/100}$$
since $k<2\varepsilon_4\log n$ provided $2\varepsilon_4 \log 10 <\frac{1}{2}$. Taking a union bound over all $\ell\in [0,k-1]$ gives that
$\P[\max_{t'\in [t+1,t+\delta n^2]}\#\{u:M_{uv}(t')=k\}>2C_110^{-k}n]\leq \frac{1}{n^{25}}$. A union bound over all $k\in [1,2\varepsilon_4\log n]$, Lemma \ref{l:lbw3} and another union bound over all vertices $v$ completes the proof of the lemma.
\end{proof}

All the lemmas in this section together imply the following Theorem.

\begin{theorem}
\label{t:weakbound}
For the rewire-to-random-* model, we have $\P(t+\delta n^2<\tau_{*},t+\delta n^2 \geq \tau'_0\mid \cf_t, t<\tau_0)\leq \frac{1}{n^{17}}$.
\end{theorem}


\section{Estimates for Strong Stopping Times}
\label{s:strongbound}

Our goal in this section is to prove that if by time $t$ the process does not reach any of the strong stopping times, then it is also unlikely that any of the strong stopping times will be hit by time $t+\delta n^2$ unless the minority opinion density drops below $\epsilon$. We shall prove this by separate analysis of each of the stopping times. In the heart of the analysis, in each case, is some estimates on how the opinions of the vertices get mixed in a short (compared to $\delta n^2$) time, which we prove by constructing a coupling of a random walk on the graph $G(t)$ with the evolving voter model dynamics.

\subsection{The Coupling Construction}
\label{s:coupling}
The dual relationship between the random walk and the voter model on a fixed graph $H$ is well known. The distribution of the voter model $X(t)$ started from $X(0)$ can be constructed by running coalescing random walks for time $t$ from each vertex and setting $X_u(t)$ to be the value of $X(0)$ at the location of the walker started from $u$ (c.f. \S 1.7, \cite{Dur06}). We prove that an analogue of that result holds in our setup. We want to say that if $t$ is small so that that graph has not changed sufficiently in the meantime, then the two distributions are not far apart. Now we formally describe the coupling.
\\

{\bf An equivalent implementation of the \emph{rewire-to-random-*} dynamics starting at time $t+1$:}
Let us condition on $\cf_{t}$. Let $N=N(t)$ be the number of edges in $G(t)$. Let the set of all labelled edges be $\{e_1,e_2,\ldots, e_N\}$. For the purpose of this subsection we shall assume that these edges are directed, i.e., the edges have two identifiable ends $e_{i}^{+}$ and $e_{i}^{-}$. When an edge is placed in a bond, we think of it as $e_i^{+}$ being placed at one vertex of the bond and $e_{i}^{-}$ being placed at the other. Suppose $e_{i}$ is placed in the bond $(u,v)$ and $e_i^{+}$ is placed at $u$ and $e_{i}^{-}$ placed at $v$. Then if a rewiring move rewires $e$ with root $u$ to the vertex $w$, then $e_i^{-}$ is placed at $w$. Consider the two independent sequences $\mathbb{RW}=\{RW_i\}_{i\geq 1}$ and $\mathbb{RL}=\{RL_i\}_{i\in \Z}$ where each $RL_i$ are chosen independently and uniformly from the set $E^*=\{e_1^+,e_1^-,e_2^+,e_2^-,\ldots , e_N^{+}, e_N^{-}\}$. The sequence $\mathbb{RW}$ is also an i.i.d. sequence where each, $RW_i=(e^*,v)$, where $e^*$ is picked uniformly from $E^*=\{e_1^+,e_1^-,e_2^+,e_2^-,\ldots , e_N^{+}, e_N^{-}\}$, and $v$ is a uniform random vertex $v$ picked from $V$ independently of $e^*$. Also let $Z_{i}$ be a sequence of i.i.d. $\mbox{Ber}\left({\frac{\beta}{n}}\right)$ variables.

We construct the equivalent formulation of the process as follows.  At time $t+i$, if $Z_i=0$, then choose the first uninspected element from $\mathbb{RW}$, let that element be $(e_j^{+},v)$. If the edge $e_j$ is not disagreeing in $G(t+i-1)$ then do nothing. Otherwise, we try to rewire the edge $e_j$ to $v$, with the root as the vertex having $e_j^{+}$. If this rewiring is not legal, (i.e., $e_j^{+}$ is already placed at $v$ in $G(t+i-1)$), then we choose the next element from $\mathbb{RW}$ and repeat the process. If $Z_i=1$, then we choose the first unused element from $\mathbb{RL}$, let that element be $e_j^{+}$. If the edge $e_j$ is not disagreeing in $G(t+i-1)$ then do nothing. Otherwise, we change the opinion of the vertex containing $e_j^{+}$. Notice that $\mathbb{RL}$ is a bi-infinite sequence but we start inspecting the elements starting from $RL_1$. It is clear from our construction that this indeed is an equivalent implementation of the \emph{rewire-to-random-*} dynamics. Let us run this dynamics starting with $G(t)$ upto $\frac{Cn^2}{\beta}$ steps. Let $\sigma$ and $\omega$ be the number of elements of $\mathbb{RL}$ and $\mathbb{RW}$ that gets inspected in the process. Notice that $\sigma$ is independent of the sequences $\mathbb{RL}$ and $\mathbb{RW}$.

{\bf Coupling with continuous time random walks:}
Now consider the following continuous time random walk on $G(t)$. Each directed edge rings at rate $\frac{\beta}{2n}$. When a directed edge rings a walker at the starting point of the edge moves along the edge. Consider the process where we start with one walker at each vertex of some arbitrary subset $W\subseteq V$, and each walker independently performs the random walk described above. We consider the following coupling between this process and the evolving voter model process described above. To start with each of the random walks are of type $A$. Now choose $T_i$ i.i.d. $exp(2N)$. At the $i$-th step, wait time $T_i$. If $Z_{Cn^2/\beta+1-i}=0$, then do nothing. If $Z_{Cn^2/\beta+1-i}=1$, then look at $RL_{\sigma+1-k(i)}$ where $k(i)=\#\{j\in [Cn^2/\beta+1-i, Cn^2/\beta]: Z_j=1\}$. If there is any walker of  type $A$ at the starting point of that edge, then that walker takes a step along that edge. If any walker of type $A$ takes a step to a vertex where there is already one or more walkers, then all the walkers become of type $B$. Type $B$ walkers do a random walk having the same waiting time distributions but using independent randomness. It is clear that the random walks are independent. Also since $\sigma$ is independent of $\mathbb{RL}$, the random walks also have correct marginals. So this is indeed a coupling as we claimed.

Let $T=\sum_{i=1}^{Cn^2/\beta}T_i$. Let us make the following definitions. For $v\in V$, let $$E^*(v)=\{e^*\in E^*: (e^*,v)=RW_i~\text{for some}~i<\omega\}.$$

\begin{definition}
\label{d:happy}
Consider the random walks described above. At time $s$, we call the walker starting at $v_0$ {\bf happy} if the following conditions are satisfied.
\begin{enumerate}
\item It is of type $A$ at time $s$.
\item None of the edges the walker has traversed have been rewired in the voter model process.
\item Let the path traversed by the walker be  $\{v_0,v_1,\ldots, v_k\}$, with the times of jump being $0=T_0^*<T_1^*<T_2^*<\ldots <T_k^*\leq s$. Then for each $\ell$, and each $e^*\in E^*(v_{\ell})$ there was no ring in $e^*$ in $[T_{\ell}^*, T_{\ell+1}^*]$.
\end{enumerate}
\end{definition}

The following lemma records the most basic useful fact about this coupling construction.

\begin{lemma}
\label{l:couplingbasic1}
Consider the coupling described above. At time $T$, if a walker starting from $v$ is {\bf happy}, then the opinion of the position of that walker at time $T$ is the same as the opinion of $v$ in the \emph{rewire-to-random-*} dynamics at time $t+\frac{Cn^2}{\beta}$.
\end{lemma}

\begin{proof}
This follows from the definition of the coupling.
\end{proof}

\begin{lemma}
\label{l:couplingbasic2}
For the coupling described as above, let $T=\sum_{i=1}^{\frac{Cn^2}{\beta}} T_i$. Then, for any $\kappa_3>0$, with exponentially high probability $\frac{2C-\kappa_3}{\beta} \leq T \leq \frac{2C+\kappa_3}{\beta}$.
\end{lemma}

\begin{proof}
The result follows from a large deviation estimate for sum of independent exponential variables.
\end{proof}

\begin{lemma}
\label{l:couplingtv}
Consider the continuous time random walk on $G(t)$ described above starting from an arbitrary vertex $v$. Let $Y(t+t')$ denote the position of the walk at time $t'$. Also let $U$ be a uniformly chosen vertex from $V$.  Then for sufficiently large $C$, we have that on $\{t<\tau'_0\}$, for all $t'\geq \frac{C}{\beta}$,
$$||Y(t+t')-U||_{TV}\leq e^{-\frac{\sqrt{C}}{1000}}.$$
\end{lemma}

\begin{proof}
We know from Lemma \ref{l:mixing} that on $\{t<\tau'_0\}$, $\lambda:=\lambda(G(t))\geq \beta \varepsilon_{14}$.
Now let $T_1$ be the time of the first jump of the walker initially at $v$. Then we know by $L^2$ Contraction Lemma (Lemma 3.26, \cite{AF02})
$$||Y(t+T_1+t^*)-U||_2 \leq e^{-t^*\lambda}||Y(t+T_1)-U||_2.$$

Now we know that on $\{t<\tau'_0\}$,

\begin{eqnarray*}
||Y(t+T_1)-U||_2^{2} &\leq & \sum _{k=1}^{\infty} \frac{9k^2}{n^2}2C_110^{-k}n^2 \\ &\leq &  18C_1\sum_{k=1}^{\infty} k^210^{-k}\\
&\leq & 50C_1.
\end{eqnarray*}

Now observe that for $t'\geq \frac{C}{\beta}$,

\begin{eqnarray*}
||Y(t+t')-U||_1 &\leq & 2||Y(t+T+t'/2)-U||_1 +\P(T_1\geq t'/2)\\
&\leq & 2 e^{-\frac{t'\lambda}{2}}||Y(t+T+t'/2)-U||_2+ \P(T_1\geq t'/2)\\
&\leq & 10\sqrt{2C_1}e^{-t'\beta \varepsilon_{14}/2}+ e^{-\frac{C\varepsilon}{16}}\\
&\leq & e^{-C\varepsilon_{14}/20}+e^{-\frac{3C\varepsilon}{40}}.
\end{eqnarray*}
The lemma follows by taking $C$ sufficiently large.
\end{proof}

Now we need the following properties of the coupling described above.

\begin{lemma}
\label{l:nonintersection}
Let $\kappa>0$ be fixed. Let $v_1,v_2,\ldots, v_{\varepsilon_{13}n}$ be given vertices in $V$. From each of the vertices we run independent discrete time simple random walks in $G(t)$ upto $20C$ steps. Then, on $\{t<\tau'_0\}$, we have for $\varepsilon_{13}$ small enough, with exponentially high probability there exists at least $(1-\kappa)\varepsilon_{13}n$ vertices among these such that the paths of random walks started from these vertices do not intersect.
\end{lemma}
\begin{proof}
Condition on $\{\cf_t, t<\tau'_0\}$. For $i=1,2,\ldots, \varepsilon_{13}n$, and $j=1,2,\ldots,4C$, let $Z_{ij}$ denote the indicator of the event that the random walk started from $i$ hits the set $V^*=\{v_1,v_2, \ldots, v_{\varepsilon_{13}n}\}$ in step $j$. Let $\cf_{i,j}$ denote the filtration generated by the random walk paths of the walks started from $v_1,v_2,\ldots, v_{i-1}$ and the first $(j-1)$ steps of the random walk started from $v_i$. Now notice that on $\{t<\tau'_0\}$, for any vertex $v$, the number of edges from $v$ to $V^*$ in $G(t)$ is at most
$$ \sum_{j: \sum_{k\geq j}{2C_110^{-k}\leq \varepsilon_{13}}} 2C_1j10^{-j}n\leq \sum_{j\geq \log(\frac{4C_1}{\varepsilon_{13}})} 2C_1j10^{-j}n\leq 25\varepsilon_{13}\log(\frac{4C_1}{\varepsilon_{13}})n \leq \frac{\varepsilon \kappa n}{400C}$$
for $\varepsilon_{13}$ sufficiently small. Since the degree of each vertex is at least $\frac{\varepsilon n}{4}$, it follows that
$$E[Z_{ij}\mid \cf_{i,j-1}]\leq \frac{\kappa}{100C}.$$
It follows from Azuma's inequality that

\begin{equation}
\label{e:azuma3}
\P\left(\sum_{i=1}^{\varepsilon_{13}n}\sum_{j=1}^{20C}\left[Z_{ij} -E[Z_{ij}\mid \cf_{i,j-1}]\right]\geq \frac{\kappa\varepsilon_{13}}{4}n\right) \leq e^{-\kappa^2\varepsilon_{13}n/640C}.
\end{equation}
Hence
\begin{equation}
\label{e:azuma4}
\P\left(\sum_{i=1}^{\varepsilon_{13}n}\sum_{j=1}^{20C} Z_{ij} \geq \frac{\kappa\varepsilon_{13}}{2}n\right) \leq e^{-\kappa^2\varepsilon_{13}n/640C}.
\end{equation}

Now let $B_i$ denote the event that the random walk starting from $v_i$ intersects a random walk started from $v_j$ for some $j<i$ at a point other than $v_i$. Let $Y_i=1_{B_i}$. Let $\mathcal{G}_i$ denote the filtration generated by the paths of random walks started from $v_1,v_2,\ldots ,v_i$. Let $C_i$ be the set of vertices visited by the first $(i-1)$ random walks except possibly $v_i$. Clearly $|C_i|\leq 25C\varepsilon_{13}n$. Arguing as before, the number of edges from any vertex $v$ to $C_i$ is at most $625C\varepsilon_{13}\log(\frac{4C_1}{25C\varepsilon_{13}})n\leq \frac{\kappa \varepsilon n}{400C}$ for $\varepsilon_{13}$ sufficiently small. By a union bound over the steps of the random walk started from $v_i$ it follows that $\E(Y_i\mid \mathcal{G}_{i-1})\leq \frac{\kappa}{4}$. Using Azuma's inequality as before we get

\begin{equation}
\label{e:azuma5}
\P\left(\sum_{i=1}^{\varepsilon_{13}n}\left[Y_i -E[Y_i\mid \mathcal{G}_{i-1}]\right]\geq \frac{\delta\varepsilon_{13}}{4}n\right) \leq e^{-\kappa^2\varepsilon_{13}n/32}.
\end{equation}
Hence
\begin{equation}
\label{e:azuma6}
\P\left(\sum_{i=1}^{\varepsilon_{13}n} Y_i \geq \frac{\kappa\varepsilon_{13}}{2}n\right) \leq e^{-\kappa^2\varepsilon_{13}n/32}.
\end{equation}

Now let $D_1$ be the set of vertices $v_i$ such that $v_i$ is hit by the random walk started from some $v_j$, $j\neq i$. Clearly
$$|D_1|\leq \sum_{i=1}^{\varepsilon_{13}n}\sum_{j=1}^{20C}Z_{ij}.$$
Also let $D_2$ be the set of vertices $v_i$ such that the random walk started from $v_i$ intersects a random walk started from $v_j$ for some $j<i$. Clearly
$$|D_2|\leq \sum_{i=1}^{\varepsilon_{13}n} Y_i.$$

It is also clear that for $i,j\in V^*\setminus (D_1\cup D_2)$, random walks started from $v_i$ and $v_j$ do not intersect. The lemma now follows from (\ref{e:azuma4}) and (\ref{e:azuma6}).
\end{proof}

\begin{lemma}
\label{l:notmanysteps}
Let $v_1,v_2,\ldots, v_{\varepsilon_{13}n}$ be given vertices in $V$. From each of the vertices we run independent continuous time random walks in $G(t)$ as described in the coupling upto time $\frac{10C}{\beta}$. Let $Z_1$ denote the number of walks that take more than $20C$ steps in this time. Then, $\P[Z_1\geq \kappa \varepsilon_{13}n\mid \cf_t, t<\tau'_0]$ is exponentially small in $n$ for $C$ sufficiently large .
\end{lemma}

\begin{proof}
Condition on $\{\cf_t, t<\tau'_0\}$. Since the number of steps taken by each random walk is independent and are stochastically dominated by a $\mbox{Poi}(5CC_2)$ variable, it follows from a large deviation estimate and $C_2\leq 2$, thar $\P(Z_1>\kappa\varepsilon_{13}n)$ is exponentially small in $n$ by taking $C$ sufficiently large.
\end{proof}

\begin{lemma}
\label{l:notrewired}
Assume the hypothesis of Lemma \ref{l:notmanysteps}. Consider the coupling of the random walks with the evolving voter model as described above. Let $Z_2$ denote the number of walkers which traverse by time $\frac{10C}{\beta}$, some edge that is rewired during the voter model process. Then $\P(Z_2\geq 3\kappa \varepsilon_{13}n\mid \cf_t, t<\tau'_0)\leq e^{-\gamma n}$ for $C$ sufficiently large, $\varepsilon_{13}=\varepsilon_{13}(C)$ sufficiently small and $\beta=\beta(C)$ sufficiently large, and for some $\gamma>0$ that does not depend on $\beta$.
\end{lemma}

\begin{proof}
Condition on $\{\cf_t, t<\tau'_0\}$. In the coupling construction, let $\omega$ be the number of entries in $\mathbb{RW}$ that were inspected. We consider the case $\omega\leq \frac{2Cn^2}{\beta}$ since the complement of this event has probability that is exponentially small in $n^2$ and can be ignored. Notice that $\mathbb{RW}$ is independent of $\mathbb{RL}$ and hence is also independent of the random walks. Let $\mathcal{H}$ denote the filtration generated by the random walk paths. Let $D_i$ denote the set of edges traversed by the random walk started at $v_i$. Let $\mathcal{D}$ denote the event that there is $J\subseteq [\varepsilon_{13}n]$ with $|J|\geq (1-2\kappa)\varepsilon_{13}n$ such that for all $j_1,j_2\in J$, $D_{j_1}\cap D_{j_2}=\emptyset$ and $|D_{j_1}|\leq 20C$. It follows from Lemma \ref{l:notmanysteps} and Lemma \ref{l:nonintersection} that $\P(\mathcal{D}^{c}\mid \cf_t, t<\tau'_0)$ is exponentially small in $n$ for $C$ sufficiently large and $\varepsilon_{13}=\varepsilon_{13}(C)$ sufficiently small. Now let us condition on $\mathcal{H}$ and $\mathcal{D}$. Since $\mathbb{RW}$ is independent of $\mathcal{H}$, for $j\in J$, we have
$$\P(\text{the edge in}~RW_{k}\in D_j~\text{for some}~k\leq \frac{2Cn^2}{\beta})\leq \frac{200C^2}{\beta}.$$

By a large deviation estimate, it follows that if $\beta$ is sufficiently large so that $\beta>\frac{2000C^2}{\kappa}$ then
$$\P(Z_2\geq 3\kappa\varepsilon_{13}n\mid \cf_t, t<\tau'_0, \mathcal{H},\mathcal{D})$$
is exponentially small in $n$. This finishes the proof of the lemma.
\end{proof}

\begin{lemma}
\label{l:happycondition}
Assume the hypothesis of Lemma \ref{l:notmanysteps}. Consider the coupling of the random walks with the evolving voter model as described above. Let $Z_3$ be the number of walks that take less than $20C$ steps up to time $\frac{10C}{\beta}$ but violates condition $3$ in Definition \ref{d:happy} at time $T$. Then $\P(Z_3>\kappa\varepsilon_{13}n)$ is exponentially small in $n$ for $\beta$ sufficiently large where the exponent does not depend on $\beta$.
\end{lemma}

\begin{proof}
Let us fix a function $q(\cdot)$ such that $q(\beta)<<\beta << q(\beta)\log q(\beta)$ as $\beta \rightarrow \infty$. Since the random walks are independent of $\mathbb{RW}$, we condition on $\mathbb{RW}$ and the following event $$\mathcal{G}=\{\omega < \frac{2Cn^2}{\beta}, |E^*(v)|\leq \frac{4Cnq(\beta)}{\beta}\forall v\in V\}.$$ By a Chernoff bound $\P(\mathcal{G})\geq 1-e^{-\gamma n}$ where $\gamma$ does not depend on $\beta$. Conditional on $\mathcal{G}$, for each $i$, the chance that the random walk started from $v_i$ violates condition $3$ in Definition \ref{d:happy} before making $20C$ many jumps is at most $\frac{200C^2 q(\beta)}{\varepsilon \beta}$. Since the events are independent for different $i$ conditional on $\mathcal{G}$, the lemma follows for $\beta$ sufficiently large.
\end{proof}

\begin{lemma}
\label{l:couplemany}
Let $v_1, v_2,\ldots , v_{\varepsilon_{13}n}$ be fixed vertices in $V$. Consider the coupling between the evolving voter model starting with $G(t)$ at time $t$, with independent continuous time random walks on $G(t)$ starting with one walker at each $v_i$ as described above.
Let us denote the position of the random walk started at $v_i$ at time $s$ by $Y_i(t+s)$. Let $\Pi$ denote the event that  there exists a time $\sigma_0\in \{C/\beta +\frac{1}{n^3}, C/\beta +\frac{2}{n^3}, \ldots, 3C/\beta\}$ and $J\subseteq [\varepsilon_{13}n]$ with $|J|\geq (1-7\kappa)\varepsilon_{13}n$, such that for each $i\in J$, opinion of $Y_i(t+\sigma_0)$ in $G(t)$ is same as $v_i(t+\frac{Cn^2}{\beta})$. Then for $C$ sufficiently large, $\varepsilon_{13}=\varepsilon_{13}(C)$ sufficiently small and $\beta=\beta(C)$ sufficiently large, we have that $\P(\Pi\mid \cf_t, t<\tau'_0)\geq 1-e^{-\gamma n/2}$ where $\gamma>0$ does not depend on $\beta$.
\end{lemma}

\begin{proof}
In this proof the value of the constant $\gamma$ may change from line to line but $\gamma$ is always a positive constant independent of $\beta$. Condition on $\{\cf_t, t<\tau'_0\}$. Let $Z_2$ denote the number of walkers at time $\frac{3C}{\beta}$ of type $B$. Then it follows from Lemma \ref{l:notmanysteps} and Lemma \ref{l:nonintersection} that $\P(Z_2>\kappa\varepsilon_{13}n)\leq e^{-\gamma n}$ for some $\gamma>0$ not depending on $\beta$. Let $Q$ denote the event that there exists $J\subseteq [\varepsilon_{13}n]$ with $|J|\geq (1-5\kappa)\varepsilon_{13}n$ such that
for all $j\in J$, the opinion of $Y_j(t+T)$ in $G(t)$ is same as $v_j(t+\frac{Cn^2}{\beta})$. It now follows from Lemma \ref{l:couplingbasic1}, Lemma \ref{l:notrewired} and Lemma \ref{l:happycondition} that for appropriate choices of $C,\varepsilon_{13}$ and $\beta$, $\P(Q^c, T<\frac{3C}{\beta}\mid \cf_t, t<\tau'_0)\leq e^{-\gamma n}$.

Now let $A$ denote the event there exist  $k\in \{1,2,\ldots ,\frac{2n^3C}{\beta}\}$ such that there are more than $\kappa\varepsilon_{13}n$ of the random walks take a step within time $[\frac{C}{\beta}+\frac{i}{n^3}, \frac{C}{\beta}+\frac{i+1}{n^3}]$. By a union bound it follows that $P(A\mid \cf_t, t<\tau'_0)\leq e^{-\gamma n}$. It follows now that,
$$\P(\Pi\mid \cf_t, t<\tau'_0)\geq 1-\P(T>\frac{3C}{\beta}\mid \cf_t, t<\tau'_0)-\P(A\mid \cf_t, t<\tau'_0)-\P(Q^c, T<\frac{3C}{\beta}\mid \cf_t, t<\tau'_0).$$
The proof of the lemma is completed using Lemma \ref{l:couplingbasic2}.
\end{proof}

\subsection{Bound for Large Cuts}
Let us fix $S,T\subseteq V$, such that $S\cap T=\emptyset$ and $S\cup T=V$ with $\varepsilon_2 n\leq |S|\leq |T|$.

\begin{proposition}
\label{p:proportionofbadedges}
For the \emph{rewire-to-random-*} dynamics, let us condition on $\{\cf_t, t<\tau'_0, N_1(t)=pn\}$. For $t'>t$, let $X_{ST}(t')$ denote the number of disagreeing edges at time $t'$. Then there exists a constant $C$ sufficiently large, and $\beta=\beta(C)$ sufficiently large such that, $$\P\left(X_{ST}(t+\frac{Cn^2}{\beta})\notin ((2p(1-p)-\varepsilon_7)N_{ST}(t+\frac{Cn^2}{\beta}), (2p(1-p)+\varepsilon_7)N_{ST}(t+\frac{Cn^2}{\beta})\mid \cf_{t}, t< \tau'_0\right)\leq e^{-\gamma n}$$
for some $\gamma>0$ that does not depend on $\beta$.
\end{proposition}

We shall need the following lemma in order to prove Proposition \ref{p:proportionofbadedges}.

\begin{lemma}
\label{l:vdisjoint}
Let $\mathcal{E}_{S,T}(t)$ denote the set of edges that have one endpoint in $S$ and another endpoint $T$ at time $t$. On $\{t<\tau'_{0}\}$, we have $|\mathcal{E}_{ST}(t)|\geq \frac{\varepsilon_2 n^2}{5}$ provided $1000\varepsilon_3< \varepsilon_2$.
Let $\kappa>0$ be fixed. Let $e_1,e_2,\ldots e_{\varepsilon_{12}n}$ be uniformly chosen edges from $\mathcal{E}_{ST}(t)$. For $\varepsilon_{12}$ sufficiently small, with exponentially high probability there exists at least $(1-\kappa)\varepsilon_{12}n$ many edges among the sample that are vertex disjoint at time $t$.
\end{lemma}

\begin{proof}
For $i=1,2,\ldots, \varepsilon_{12}n$, let $A_i$ denote the event that $e_i$ is not vertex disjoint with $e_1,e_2,\ldots , e_{i-1}$. Let $Z_i=1_{A_i}$. Let $\cf_{i-1}$ denote the filtration generated by $e_1,e_2,\ldots, e_{i-1}$. Then it is clear from the assumption on the graph that $$\E(Z_i\mid \cf_{i-1})\leq \frac{10C_2n+10\varepsilon_{3}\log n}{\varepsilon_2 n^2} \varepsilon_{12}n\leq \frac{12C_2 \varepsilon_{12}}{\varepsilon_2}.$$
Also notice that $|Z_i-\E(Z_i\mid \cf_{i-1})|\leq 1$ and hence Azuma's inequality yields

\begin{equation}
\label{e:azuma1}
\P\left(\sum_{i=1}^{\varepsilon_{12}n}\left[Z_i -E[Z\mid \cf_{i-1}]\right]\geq \frac{\kappa\varepsilon_{12}}{2}n\right) \leq e^{-\kappa^2\varepsilon_{12}n/8}.
\end{equation}

It follows that

\begin{equation}
\label{e:azuma2}
\P\left(\sum_{i=1}^{\varepsilon_{12}n}Z_i \geq \frac{12C_2\varepsilon_{12}^2n}{\varepsilon_2}+ \frac{\kappa\varepsilon_{12}}{2}n\right) \leq e^{-\kappa^2\varepsilon_{12}n/8}.
\end{equation}

By choosing $\varepsilon_{12}$ sufficiently small such that $\frac{12C_2\varepsilon_{12}}{\varepsilon_2}\leq \frac{\kappa}{2}$ completes the proof.
\end{proof}

\begin{proof}[Proof of Proposition \ref{p:proportionofbadedges}]
Let us condition on $\{\cf_{t}, t<\tau'_0, N_1(t)=pn\}$. Let $N_{ST}(t)$ be the number of edges in $G(t)$ with one endpoint in $S$ and another endpoint in $T$. Let  $\{e_1,e_2,\ldots , e_{N_{ST}(t)}\}$ denote the set of those edges. Let $X_i$ be the indicator that endpoints of $e_i$ in $G(t)$ are disagreeing in $G(t+\frac{Cn^2}{\beta})$. Since in $Cn^2/\beta$ steps, at most $Cn^2/\beta$ edges can be rewired it follows that for $\beta$ sufficiently large, it suffices to prove that

%

$$\P(\frac{1}{N_{ST}(t)}\sum_{i=1}^{N_{ST}(t)}X_i \notin ((2p(1-p)-\varepsilon_7/2), (2p(1-p)+\varepsilon_7/2)\mid \cf_{t}, t< \tau'_0)\leq e^{-\gamma n}.$$

%
%
%

Let $J$ be a set of size $\varepsilon_{12}n$ where each element is an independent uniform sample from $[N_{ST}(t)]$. Clearly by Hoeffding's inequality,

$$\P\left(\mid \frac{1}{\varepsilon_{12}n}\sum_{j\in J} X_j - \frac{1}{N_{ST}(t)}\sum_{i=1}^{N}X_i \mid \geq \varepsilon_7/4 \right)\leq e^{-\frac{\varepsilon_{7}^2\varepsilon_{12}n}{32}}.$$

So it suffices for us to prove that with probability at least $1-e^{-2\gamma n}$,

$$\frac{1}{\varepsilon_{12}n}\sum_{j\in J} X_j\in  (2p(1-p)-\varepsilon_7/4, 2p(1-p)+\varepsilon_7/4).$$

Choose $\varepsilon_{12}$ sufficiently small, and set $\varepsilon_{13}=2\varepsilon_{12}(1-\kappa)$ so that the conclusions of Lemma \ref{l:vdisjoint} and Lemma \ref{l:couplemany} are satisfied. Let $\mathcal{H}_1$ denote the event that there is a subset $J^*\subseteq J$ with $|J^*|=(1-\kappa)\varepsilon_{12}n$ such that endpoints of $e_j$ are disjoint for all $j\in J^*$. It follows from Lemma \ref{l:vdisjoint} that $\P(\mathcal{H}\mid \cf_t, t<\tau'_0)\geq 1-e^{-100\gamma n}$. Condition on $\mathcal{H}_1$ and $J^*$. Let $v_1,v_2,\ldots, v_{\varepsilon_{13}n}$ be  endpoints of edges of $J^*$. By choosing $100\kappa < \varepsilon \varepsilon_{7}$ it follows that it suffices to prove
\begin{equation}
\label{e:reducedlc}
\P[\frac{2}{\varepsilon_{13}n}\sum_{j\in J^*} X_j\notin  (2p(1-p)-\varepsilon_7/8, 2p(1-p)+\varepsilon_7/8)\mid \cf_t, t<\tau'_0, N_1(t)=pn, \mathcal{H}_1, J]\leq e^{-10\gamma n}.
\end{equation}

Consider the coupling described in \S~\ref{s:coupling}. Fix $j\in J^*$, let $v_{j_1}$ and $v_{j_2}$ be endpoints of $e_j$ in $G(t)$. For  $\tilde{\sigma}\in \Sigma=\{C/\beta +\frac{1}{n^3}, C/\beta +\frac{2}{n^3}, \ldots, 3C/\beta\}$ let $U_{j,\tilde{\sigma}}$ denote the indicator that the position of the coupled random walks started from $v_{j_1}$ and $v_{j_2}$ at time $\tilde{\sigma}$ have different opinions in $G(t)$. Clearly, for a fixed $\tilde{\sigma}$, for all $j\in J^*$, $U_{j,\tilde{\sigma}}$ are conditionally independent. Also, it follows from Lemma \ref{l:couplingtv} that $\E(U_{j,\tilde{\sigma}}\mid \cf_t, t<\tau'_0, N_1(t)=pn)\in [2p(1-p)-\varepsilon_7/32,2p(1-p)+\varepsilon_7/32]$.

A standard Hoeffding bound now shows that conditional on $\mathcal{G}=\{\cf_t, t<\tau'_0, N_1(t)=pn, \mathcal{H}_1,J^*\}$, with probability at least $1-e^{-20\gamma n}$,
\begin{equation}
\label{e:reducedlc3}
\frac{2}{\varepsilon_{13}n}\sum_{j\in J^*} U_{j,\tilde{\sigma}}\in  (2p(1-p)-\varepsilon_7/16, 2p(1-p)+\varepsilon_7/16).
\end{equation}
By taking a union bound over all possible values of $\tilde{\sigma}$, it follows that that above holds for all $\tilde{\sigma}$ in $\Sigma$ with  conditional probability at least $1-e^{-15\gamma n}$.

By observing that by Lemma \ref{l:couplemany}, we have, conditional on $\mathcal{G}$ there exists $\tilde{\sigma}\in \Sigma$
$$\frac{2}{\varepsilon_{13}n}\sum_{j\in J^*} |U_{j,\tilde{\sigma}}-X_j|\leq 12\kappa \leq \varepsilon_7/32$$
with probability at least $1-e^{-15\gamma n}$. This and the previous observation implies (\ref{e:reducedlc}) and the proof of the proposition is complete.
\end{proof}

The following proposition follows along the same lines as Proposition \ref{p:proportionofbadedges} and we shall omit the proof.


\begin{proposition}
\label{p:proportionofbadedgesss}
Fix $S\subseteq V$ with $|S|\geq \varepsilon_2 n$.  For the \emph{rewire-to-random-*} dynamics, let us condition on $\{\cf_t, t<\tau'_0, N_1(t)=pn\}$. For $t'>t$, let $X_{SS}(t')$ denote the number of disagreeing edges at time $t'$. Then there exists a constant $C$ sufficiently large, and $\beta=\beta(C)$ sufficiently large such that, $$\P\left(X_{SS}(t+\frac{Cn^2}{\beta})\notin ((2p(1-p)-\varepsilon_7)N_{SS}(t+\frac{Cn^2}{\beta}), (2p(1-p)+\varepsilon_7)N_{SS}(t+\frac{Cn^2}{\beta})\mid \cf_{t}, t< \tau'_0\right)\leq e^{-\gamma n}$$
for some $\gamma>0$ that does not depend on $\beta$.
\end{proposition}



Condition on $\{\cf_{t}, t<\tau_0, N_1(t)=pn\}$. Let $D^*(t')$ denote the event that $N_1(t'')\in ((p-\varepsilon_7)n, (p+\varepsilon_7)n)$ for all $t''\in [t+1, t']$.  Fix $S$ and $T$ as in Proposition \ref{p:proportionofbadedges}. For $t'\in [t+1,t+\delta n^2]$, let us define events
$\mathcal{A}^{t'}_{ST}$, $\mathcal{A}_{SS}^{t'}$ and $\mathcal{A}_{T,T}^{t'}$ as follows.
%
$$\mathcal{A}^{t'}_{ST}=\left\{X_{ST}(t')\in (2p(1-p)-2\varepsilon_7, 2p(1-p)+2\varepsilon_7)N_{ST}(t')\right\}.$$
$$\mathcal{A}^{t'}_{SS}=\left\{X_{SS}(t')\in (2p(1-p)-2\varepsilon_7, 2p(1-p)+2\varepsilon_7)N_{SS}(t')\right\}.$$
$$\mathcal{A}_{TT}^{t'}=\left\{X_{TT}(t')\in (2p(1-p)-2\varepsilon_7, 2p(1-p)+2\varepsilon_7)N_{TT}(t')\right\}.$$
Finally, let us define
$$\mathcal{A}^{t'}= \mathcal{A}_{SS}^{t'}\cap \mathcal{A}_{ST}^{t'}\cap \mathcal{A}_{TT}^{t'}.$$
Let $Z_{t'}$ be the indicator of $\overline{\mathcal{A}^{t'}}$.


\begin{lemma}
\label{l:propcount}
Let $\mathcal{G}$ denote the conditioning $\mathcal{G}=\{\cf_t, t<\tau_0, N_1(t)=pn\}$. Then we have

$$\P\left(\frac{1}{\delta n^2}\sum_{t'=t+1}^{t+\delta n^2} Z_{t'}\geq \varepsilon_{15}, t+\delta n^2<\tau_0', D^*(t+\delta n^2) \mid \mathcal{G}\right)\leq e^{-h(\beta)n}$$

where $h(\beta)$ can be made arbitrarily large by taking $\beta$ sufficiently large.
\end{lemma}

\begin{proof}
For $i=1,2,\ldots ,\frac{Cn^2}{\beta}$, and for $k=1,2, \ldots \frac{\delta \beta}{C}-1$ let $t_{i,k}=t+i+k\frac{Cn^2}{\beta}$. Let

$$W_i=\sum_{k=1}^{\frac{\delta \beta}{C}-1}Z_{t_{i,k}}.$$

From Proposition \ref{p:proportionofbadedges} and Proposition \ref{p:proportionofbadedgesss} it follows that for a fixed $i$, $\P[Z_{t_{i,k+1}}=1\mid \cf_{t_{i,k}}, t_{i,k}<\tau'_0, D^*(t_{i,k})]\leq e^{-\gamma n/5}$. Using a Chernoff bound it follows that

\begin{eqnarray}
\label{e:timeprop1}
\P\left[\frac{C}{\delta\beta} W_{i} \geq \varepsilon_{15}/2, t+\delta n^2\leq \tau'_0, D^{*}(t+\delta n^2) \mid \mathcal{G}\right] &\leq & \exp(-\varepsilon_{15}\delta\beta\log({\frac{\varepsilon_{15}e^{-\gamma n/5}}{2}})/4C)\nonumber \\
&\leq & (\frac{2}{\varepsilon_{15}})^{1/4C}\exp(-\varepsilon_{15}\delta\beta\eta n/20C).
\end{eqnarray}

Taking a union bound over all $i$ and choosing $\beta$ sufficiently large so that $\frac{C}{\delta\beta}\leq \varepsilon_{15}/2$
it follows that

$$\P\left(\frac{1}{\delta n^2}\sum_{t'=t+1}^{t+\delta n^2} Z_{t'}\geq \varepsilon_{15}, t+\delta n^2<\tau' \mid \cf_t \right)\leq \frac{Cn^{2}}{\beta}(\frac{2}{\varepsilon_{15}})^{1/4C}\exp(-\varepsilon_{15}\delta\beta\gamma n/20C).$$

Taking $\beta$ sufficiently large completes the proof of the lemma.
\end{proof}

\begin{proposition}
\label{p:bigcut}
Let $\mathcal{G}$ denote the conditioning $\mathcal{G}=\{\cf_t, t<\tau_0,N_1(t)=pn\}$. Let $S,T$ be as above. Then we have, $$\P(K_{ST}(t+\delta n^2)> \varepsilon_3^2, t+\delta n^2<\tau'_0, D^*(t+\delta n^2) \mid \mathcal{G})\leq \exp (-h(\beta)n)$$

where $h(\beta)$ can be made arbitrarily large by taking $\beta$ sufficiently large.
\end{proposition}

\begin{proof}
For $s\in [t,t+\delta n^2-1]$, let $\cf_{s}$ denote the filtration generated by the process up to time $s$. Conditioned on $\cf_{s}$ the transition rule for the evolution of $(N_{SS}(s), N_{TT}(s))$ is given by the following.

\begin{equation*}
(N_{SS}(s+1), N_{TT}(s+1))=
\begin{cases}
(N_{SS}(s)+1, N_{TT}(s))~~w.p. ~~  \frac{X_{ST}(s)}{N(s)}\frac{1}{2}\frac{|S|-1}{n-1}(1-\frac{\beta}{n}) \\
(N_{SS}(s), N_{TT}(s)+1)~~ w.p. ~~  \frac{X_{ST}(s)}{N(s)}\frac{1}{2}\frac{|T|-1}{n-1}(1-\frac{\beta}{n}) \\
(N_{SS}(s)-1, N_{TT}(s))~~ w.p.~~ \frac{X_{SS}(s)}{N(s)}\frac{|T|}{n-1}(1-\frac{\beta}{n})  \\
(N_{SS}(s), N_{TT}(s)-1)~~ w.p.~~ \frac{X_{TT}(s)}{N(s)}\frac{|S|}{n-1}(1-\frac{\beta}{n}) \\
(N_{SS}(s), N_{TT}(s)) ~~~~~~    \text{otherwise}.

\end{cases}
\end{equation*}

For this proof, let us write $\Delta=2p(1-p)(1-\frac{\beta}{n})$ and $\varepsilon_{8}=\varepsilon_{7}/2p(1-p)$. For $n$ sufficiently large and on $Z_{s}=0$, we have

\begin{equation*}
(N_{SS}(s+1), N_{TT}(s+1))=
\begin{cases}
(N_{SS}(s)+1, N_{TT}(s))~~ w.p. ~~  \in \Delta\frac{N_{ST}(s)}{N(s)}\frac{1}{2}\frac{|S|}{n}(1\pm 3\varepsilon_8) \\
(N_{SS}(s), N_{TT}(s)+1)~~ w.p. ~~ \in \Delta\frac{N_{ST}(s)}{N(s)}\frac{1}{2}\frac{|T|}{n}(1\pm 3\varepsilon_8) \\
(N_{SS}(s)-1, N_{TT}(s)) ~~ w.p. ~~ \in \Delta\frac{N_{SS}(s)}{N(s)}\frac{|T|}{n}(1\pm 3\varepsilon_8)  \\
(N_{SS}(s), N_{TT}(s)-1)~~ w.p. ~~ \in \Delta\frac{N_{TT}(s)}{N(s)}\frac{|S|}{n}(1\pm 3\varepsilon_8) \\
\end{cases}
\end{equation*}

Doing a change of variable $W_{S}(s)=\frac{N_{SS}(s)-|S|^2/4}{N}$ and $W_{T}(s)=\frac{N_{TT}(s)-|T|^2/4}{N}$ it follows that on $\{Z_{s}=0\}$, we have
\begin{equation*}
(W_{S}(s+1), W_{T}(s+1))=
\begin{cases}
(W_{S}(s)+1/N, W_{T}(s))~~ w.p. ~~  \in\Delta\left(1-W_S(s)-W_{T}(s)-\frac{|S|^2+|T|^2}{4N}\right)\frac{|S|}{2n}\pm 3\varepsilon_7 \\
(W_{S}(s), W_{T}(s)+1/N)~~ w.p. ~~ \in\Delta\left(1-W_S(s)-W_{T}(s)-\frac{|S|^2+|T|^2}{4N}\right)\frac{|T|}{2n}\pm 3\varepsilon_7 \\
(W_{S}(s)-1/N, W_{T}(s)) ~~ w.p. ~~ \in \Delta\left(W_{S}(s)+\frac{|S|^2}{4N}\right)\frac{|T|}{n}\pm 3\varepsilon_7  \\
(W_{S}(s), W_{T}(s)-1/N)~~ w.p. ~~ \in \Delta\left(W_T(s)+\frac{|T|^2}{4N}\right)\frac{|S|}{n} \pm 3\varepsilon_7 \\
\end{cases}
\end{equation*}
It follows that on $\{Z_s=0\}$,

\begin{eqnarray*}
\E[W_S(s+1)^2+W_{T}(s+1)^2\mid \cf_s] &\leq & W_S(s)^2+W_{T}(s)^2 \\
&+&\frac{2W_S(s)\Delta}{N}\left[\frac{|S|}{2n}-\frac{|S|n}{8N}-W_S(s)(\frac{|S|+2|T|}{2n})-W_T(s)\frac{|S|}{2n}\right]\\
&+& \frac{2W_T(s)\Delta}{N}\left[\frac{|T|}{2n}-\frac{|T|n}{8N}-W_T(s)(\frac{2|S|+|T|}{2n})-W_T(s)\frac{|T|}{2n}\right]\\
 &+&\frac{25\varepsilon_7}{n^2}+o(\frac{1}{n^2})\\
 &\leq & W_S(s)^2+W_{T}(s)^2\\
 &-& \frac{4\Delta}{n^2}\left[\frac{1}{4}(W_S(s)^2+W_{T}(s)^2)+\frac{1}{4}(W_s(s)+W_T(s))^2\right]\\
 &-& \frac{4\Delta}{n^2}\left[\frac{|T|}{2n}W_S(s)^2+\frac{|S|}{2n}W_T(s)^2\right] +\frac{32\varepsilon_7}{n^2}\\
 &\leq & W_S(s)^2+W_{T}(s)^2-\frac{\Delta}{n^2}\left[(W_S(s)^2+W_{T}(s)^2)-32\varepsilon_8\right]
\end{eqnarray*}

Hence, we have

$$\E[K_{ST}(s+1)-K_{ST}(s)\mid \cf_s, Z_{s}=0]\leq -\frac{\Delta}{n^2}(K_{ST}(s)-32\varepsilon_{8}).$$

In particular, on $\{Z_{s}=0\}\cap \{K_{ST}(s)\geq \varepsilon_3^2/2\}$, we have
$$\E[K_{ST}(s+1)-K_{ST}(s)\mid \cf_s]\leq -\frac{\Delta}{4n^2}\varepsilon_3^2$$\
by choosing $\varepsilon_3^2\geq 128\varepsilon_8$.

Now notice that $K_{ST}(s+1)-K_{ST}(s)\leq \frac{16}{n^2}$. Let $\mathcal{C}$ denote the event
$$\mathcal{C}=\{\min_{s\in [t,t+\delta n^2]}K_{ST}\leq \varepsilon_{3}^2/2\}.$$
Let $\mathcal{D}$ denote the event
$$\mathcal{D}=\left\{\sum_{s=t}^{t+\delta n^2-1}\E[K_{ST}(s+1)-K_{ST}(s)\mid \cf_s]>\delta(16\varepsilon_{15}-(1-\varepsilon_{15})\Delta\varepsilon_{3}^2/4)\right\}.$$
It follows from Lemma \ref{l:propcount} that
$$\P[\mathcal{C}^c, \mathcal{D}, D^*(t+\delta n^2), t+\delta n^2 <\tau'_0\mid \mathcal{G}]\leq \exp (-h(\beta)n)$$

where $h(\beta)$ can be made arbitrarily large by choosing $\beta$ sufficiently large.

Now an application of Azuma-Hoeffding inequality gives
\begin{equation}
\label{e:azuma7}
\P\left[\sum_{s=t}^{t+\delta n^2-1} K_{ST}(s+1)-K_{ST}(s)-\E(K_{ST}(s+1)-K_{ST}(s)\mid \cf_s)\geq \frac{\delta\varepsilon_{3}^2}{10}\mid \mathcal{G}\right]\leq \exp (-\frac{\delta\varepsilon_{3}^4n^2}{204800}).
\end{equation}

It follows that if $\varepsilon_{15}$ is chosen sufficiently small so that $\varepsilon \varepsilon_3^{2}> \frac{64\varepsilon_{15}}{1-\varepsilon_{15}}$, then we have,

$$\P[K_{ST}(t+\delta n^2)> K_{ST}(t), \mathcal{C}^c, D^*(t+\delta n^2) t+\delta n^2< \tau'_0\mid \mathcal{G}]\leq \exp (-h(\beta)n).$$

Observe that by choosing $\delta$ sufficiently small ($16 \delta < \varepsilon_3^2$), it follows that on $\mathcal{C}$, $K_{ST}(t+\delta n^2)< \varepsilon_3^2$. Hence

$$\P[K_{ST}(t+\delta n^2)> \varepsilon_{3}^2, D^*(t+\delta n^2), t+\delta n^2< \tau'_0\mid \mathcal{G}]\leq \exp (-h(\beta)n).$$

This completes the proof of the proposition.
\end{proof}

Now we are ready to prove the main result of this subsection.

\begin{theorem}
\label{t:bigcut}
We have for all $t\geq \delta n^2$,
$$\P[t+\delta n^2\geq \tau_2, t+\delta n^2<\tau_*, t<\tau_0]\leq \frac{1}{n^{14}}.$$
\end{theorem}

\begin{proof}
Let $\tilde{D}(t+\delta n^2)$ be defined as follows.

$$\tilde{D}(t+\delta n^2)=\{\forall t'\in [t+1, t+\delta n^2]: |N_1(t')-N_1(t)|\leq \varepsilon_{7}n\}.$$

By taking a union bound over all cuts $S,T$ such that $\varepsilon_2 n\leq |S|\leq |T|$ and using Proposition \ref{p:bigcut} we get that

$$\P[L(t+\delta n^2)\geq \varepsilon_3^{2}, \tilde{D}(t+\delta n^2), t+\delta n^2 <\tau_0' \mid \cf_t, t<\tau_0]\leq 2^n\exp(-h(\beta)n)\leq \frac{1}{n^{18}}$$
by taking $\beta$ sufficiently large. It follows by a random walk estimate that $\P[\tilde{D}(t+\delta n^2)\mid \cf_t, t<\tau_0]$ is exponentially close to $1$ and hence we have,

$$\P[L(t+\delta n^2)\geq \varepsilon_{3}^2, t+\delta n^2<\tau'_0\mid t<\tau, \cf_t]\leq \frac{2}{n^{18}}.$$

By Theorem \ref{t:weakbound}, we know that $\P[t+\delta n^2\geq \tau'_0, t+\delta n^2< \tau_{*}\mid \cf_t, t<\tau_0]\leq \frac{1}{n^{17}}$ and hence

$\P[L(t+\delta n^2)\geq \varepsilon_3^2, t+\delta n^2< \tau_*\mid \cf_t , t<\tau_0]\leq \frac{2}{n^{17}}$.

Now as $\{t<\tau\}\subseteq \{t-s<\tau\}$ and $\{t+\delta n^2 <\tau_{*}\}\subseteq \{t+\delta n^2-s <\tau_{*}\}$ for each $s\geq 0$, by taking a union over $s\in [0,\delta n^2 -1]$ we get that

$$\P[\tau_2\leq t+\delta n^2, \tau_{*}< t+\delta n^2, t<\tau_0]\leq \frac{1}{n^{14}}.$$
This completes the proof of the theorem.
\end{proof}

\subsection{Edge Multiplicity}
In this subsection we consider the strong stopping time associated with edge multiplicities, i.e., $\tau_3$. We need the following lemma.

\begin{lemma}
\label{l:probagreeing}
Let $u,v$ be two fixed vertices in $V$. Let $X_{uv}(t')$ be the indicator that $u$ and $v$ are disagreeing in $G(t')$. Then we have, $\E(X_{uv}(t+\frac{Cn^2}{\beta})\mid \cf_{t}, t<\tau'_0)\geq \frac{\varepsilon}{2}$.
\end{lemma}

\begin{proof}
To prove this consider the coupling of the evolving voter model with independent continuous time random walk started from $u$ and $v$ as described in \S~\ref{s:coupling}.  Notice that the chance that the random walks intersect upto time $\frac{3C}{\beta}$ is $o(1)$ as $n\rightarrow \infty$. Also notice that the chance that either of the walk traverses any edge that was rewired can be made less than $\frac{\varepsilon}{100}$ by choosing $\beta$ sufficiently large. Let $Y_{uv}$ denote the indicator that the positions of the random walks started from $u$ and $v$ after time $T$ are disagreeing in $G(t)$. Clearly $P(X_{uv}(t+\frac{Cn^2}{\beta})\neq Y_{uv})\leq \varepsilon/4$ using Lemma \ref{l:couplingbasic1}. Also let $Y^*_{uv}$ be the indicator that the position of random walks started from $u$ and $v$ after time $\frac{C-\varepsilon/100}{\beta}$ are disagreeing in $G(t)$. Using Lemma \ref{l:couplingbasic2} it follows that $\P(Y_{uv}\neq Y^*_{uv})\leq \varepsilon/4$. The result follows by noticing that Lemma \ref{l:couplingtv} implies that for $C$ sufficiently large $\E(Y^*_{uv})\geq \varepsilon$.
\end{proof}
%

%
%
%
%
%

\begin{lemma}
\label{l:multbadprounion}
Let $u,v\in V$ be two vertices in $V$. Fix $t>\varepsilon_{16}n^2\log n$. For $\varepsilon_{16}n^2\log n<t'<t$, let $A_{t'}$ denote the event that there exists $\mathcal{T}\in \{1,2,\ldots, t-t'\}$ such that 
$$\#\{s\in \{\mathcal{T}, \mathcal{T}+1, \ldots , \mathcal{T}+t'\}:u(s)\neq v(s)\}\leq t'\varepsilon/4.$$
Then we have 
$$\P\left[\cup A_{t'}, t<n^4\wedge \tau'_0\right]\leq \frac{1}{n^{r(\beta)}}$$
where $r(\beta)$ can be made arbitrarily large by taking $\beta$ sufficiently large.
\end{lemma}

\begin{proof}
Fix $\varepsilon_{16}n^2\log n<t'<t\wedge n^4$ and $\mathcal{T}\in \{1,2,\ldots t-t'\}$. For $t''\in \{\mathcal{T}, \mathcal{T}+1,\ldots, \mathcal{T}+t'\}$ it follows from Lemma \ref{l:probagreeing} that on $\{t''<\tau'_0\}$, $\P[u(t''+Cn^2/\beta)\neq  v(t''+Cn^2/\beta)\mid \cf_{t''}]\geq \varepsilon/2$ for $\beta$ sufficiently large. It follows using a Chernoff's bound that for each $i=1,2,\ldots, Cn^2/\beta$,
$$\P\left[\#\{t''\in \{\mathcal{T}+i+kCn^2/\beta: k\in [\beta t'/n^2C]\}: u(t'')\neq  v(t'')\}\leq \frac{\beta t'\varepsilon}{4Cn^2}, t<\tau'_0\right]\leq  \exp (-\frac{\beta t' \varepsilon}{12Cn^2}).$$
For all $t'>\varepsilon_{16}n^2\log n$, we have the right hand side of the above inequality is at most $(\frac{1}{n})^{\beta \varepsilon_{16}\varepsilon/12C}$
and it follows by taking a union bound over all $i\in [Cn^2/\beta]$ and all $\mathcal{T}\in \{1,2,\ldots t-t'\}$ that 
$$\P[A_{t'},t<n^4\wedge \tau'_0]\leq \frac{1}{n^{r'(\beta)}}$$ where $r'(\beta)$ can be made sufficiently large by choosing $\beta$ to be sufficiently large. The lemma now follows by taking union bound over $\varepsilon_{16}n^2\log n<t'<t\wedge n^4$.
\end{proof}

Now we define the following family of random walks which we couple with the \emph{rewire-to-random-*} dynamics as follows, $X^{s}(\cdot)$  indexed by $s\in \{1,2,\ldots, t\}$ with each starting from $K>0$ (i.e., $X^{s}(0)=K~ \forall s$) with transition probabilities as described below.

\begin{equation*}
X^s(h+1)=
\begin{cases}
X^{s}(h)+1 ~~ w.p. ~~  \frac{9C_2}{n^2} \\
X^{s}(h)-1 ~~ w.p. ~~  \frac{K}{n^2}~\text{if}~O_u(s+h)\neq O_v(s+h) \\
X^{s}(h) ~~~ ~otherwise.
\end{cases}
\end{equation*}

The following lemma is immediate by comparing one step transition probabilities of $M_{uv}(t)$ and $X^{s}(t-s)$.

\begin{lemma}
\label{l:multstocdom}
Let $M_{uv}^*(t)=\max_{t'\in[1,t]}M_{uv}(t')$ and $X^{*}(t)=\max_{s,h: s+h\leq t} X^{s}(t)$. Then we have, on $\{t<\tau'_0\}$, $M_{uv}^{*}(t)\preceq X^{*}(t)$ where $\preceq$ denotes stochastic domination.
\end{lemma}

From the previous lemma, we deduce the following.

\begin{lemma}
\label{l:multmart}
We have
$\P[M_{uv}^*(t)>\varepsilon_4 \log n , t<\tau'_0\wedge n^4]\leq \frac{1}{n^{10}}$.
\end{lemma}

\begin{proof}
By Lemma \ref{l:multstocdom} it suffices to prove the inequality in the statement with $M_{uv}^*(t)$ replaced by $X^{*}(t)$.
Let $\mathcal{C}$ denote the following event.
$$\mathcal{C}=\{\forall T\in [1,t], t'> \varepsilon_{16}n^2\log n~ \#\{s\in [T,T+t']:
u(s)\neq v(s)\}\geq t'\varepsilon/4\}.$$

Then we have that for all $t-s>t'>\varepsilon_{16}n^2\log n$,

$$\E(e^{\lambda X^{s}(t')}1_{\mathcal{C}\cap \{t<\tau'_0\}})\leq e^{\lambda K}\left(1+\frac{9C_2}{n^2}(e^{\lambda}-1)\right)^{t'(1-\varepsilon/4)}\left(1+\frac{9C_2}{n^2}(e^{\lambda}-1)+\frac{K}{n^2}(1-e^{-\lambda})\right)^{t'\varepsilon/4}.$$

Fix $\lambda$ large enough such that $\lambda\varepsilon_{4}>20$. Choosing $K$ sufficiently large depending on $\lambda$ and $\varepsilon$, it follows that

$$(1+\frac{9C_2}{n^2}(e^{\lambda}-1))^{1-\varepsilon/4}(1+\frac{9C_2}{n^2}(e^{\lambda}-1)+\frac{K}{n^2}(1-e^{-\lambda}))^{\varepsilon/4}< 1$$
and hence

$$\E(e^{\lambda X^{s}(t')} 1_{\mathcal{C}\cap \{t<\tau'_0\}})\leq e^{\lambda K}.$$

By Markov's inequality it follows that

$$\P(\{X^{s}_{t'}>\varepsilon_4\log n\}\cap \mathcal{C}\cap \{t<\tau'_0\})\leq e^{-\lambda(K-\varepsilon_4\log n)}\leq (\frac{1}{n})^{19}.$$

For $t'<\varepsilon_{16}n^2\log n$, $X^{s}(t')-K$ is stochastically dominated by a $\mbox{Bin}(\varepsilon_{16}n^2\log n, \frac{9C_2}{n^2})$ variable. Using a Chernoff bound in this case, we get for $n$ sufficiently large
$$\P[X^s(t')\geq \varepsilon_4\log n]\leq e^{-\varepsilon_4\log n\log(\varepsilon_{4}/9C_2\varepsilon_{16})/2} \leq (\frac{1}{n})^{19} $$
by choosing $\varepsilon_{16}$ sufficiently small such that $\varepsilon_4\log(\varepsilon_{4}/9C_2\varepsilon_{16})>38.$

By taking a union bound over all $s,t'$ it follows that

$$\P[X^{*}(t)>\varepsilon_4 \log n, \mathcal{C}, \{t<\tau'_0\}]\leq (\frac{1}{n})^{11}.$$

The result now follows from Lemma \ref{l:multbadprounion}.
\end{proof}

\begin{theorem}
\label{t:indmultbound}
Let $t<n^4$. Then $\P[t\geq \tau_3, t< \tau'_0]\leq \frac{1}{n^4}$.
\end{theorem}

\begin{proof}
The theorem follows from using Lemma \ref{l:multmart}, taking a union bound over all $(u,v) \in V^{(2)}$.
\end{proof}

\subsection{Degree Estimate}
In this section we prove the following theorem.
\begin{theorem}
\label{t:degsbound}
We have for all $t\geq \delta n^2$, $\P[t+\delta n^2\geq \tau_5, t+\delta n^2<\tau_{*}, t<\tau]\leq \frac{1}{n^{14}}$.
\end{theorem}

We start with the following lemma.

\begin{lemma}
\label{l:degcoupling}
Let $\kappa_2>0$ be fixed. Let us condition on $\{\cf_{t_1}, t_1<\tau', N_{*}(t)=pn\}$. Let $v$ be a fixed vertex in $V$. Let $X_v(t')$ denote the number of disagreeing edges incident to $v$ at time $t'$. Then for sufficiently large $C$ and sufficiently large $\beta=\beta(C)$, at time $t_2=t_1+\frac{Cn^2}{\beta}$ we have $\P[X_v(t_2)\notin (p(1-\kappa_2), 1-p(1-\kappa_2))D_v(t_2)\mid \cf_{t_1}, \tau'_0>t_1]\leq e^{-\varepsilon_{20}n}$ for some constant $\varepsilon_{20}>0$.
\end{lemma}

\begin{proof}
The proof of this lemma goes along the same lines as that of Proposition \ref{p:proportionofbadedges}. We shall therefore only give the sketch of the steps. Let the edges incident to $v$ at time $t_1$ be $\{e_1,e_2,\ldots, e_{D_v(t_1)}\}$. Let $Y_i$ denote the indicator that the endpoints of $e_i$ are disagreeing in $G(t_2)$. It follows by taking $\beta$ sufficiently large that it suffices to prove that

$$\frac{1}{D_v(t_1)}\sum_{i=1}^{D_v(t)} Y_i \in (p(1-\kappa_2/2),1-p(1-\kappa_2/2))$$
with exponentially high probability. To this end, we choose a subset of these edges of size $\varepsilon_{12}n$. Condition on a subset $J$ of these edges of size at least $\varepsilon_{13}n=(1-\frac{\kappa_2}{100})\varepsilon_{12}n$ which correspond to distinct bonds in $V^{(2)}$. Arguing as in the proof of Proposition \ref{p:proportionofbadedges}, we see that it suffices to show that, conditionally,
$$\frac{1}{\varepsilon_{13}n}\sum_{j\in J} Y_j \in (p(1-\kappa_2/50),1-p(1-\kappa_2/50))$$
with exponentially high probability.

Consider the coupling described in \S~\ref{s:coupling} of the \emph{rewire-to-random-*} dynamics started with $G(t_1)$ with independent random walks started from $v_j$ where $e_j$ is placed in the bond $(v,v_j)$ is $G(t_1)$. Let for $j\in J$ and $\tilde{\sigma}>\frac{C}{\beta}$, $U^0_{j,\tilde{\sigma}}$ and $U^1_{j,\tilde{\sigma}}$ denote the indicators that the position of the random walk started from $v_j$ has the opinion $0$ and opinion $1$ respectively at time $\tilde{\sigma}$. It follows by taking $C$ sufficiently large that $\E(U^0_{j,\tilde{\sigma}}), \E(U^1_{j,\tilde{\sigma}})\in (p(1-\kappa_2/100),1-p(1-\kappa_2/100))$. Now the proof is completed arguing as in the proof of Proposition \ref{p:proportionofbadedges}.
\end{proof}

\begin{lemma}
\label{l:deginterval}
With the notation as in Lemma \ref{l:degcoupling}, let $A_t$ denote the event that for some $t'$ with $t'\in [t+\frac{Cn^2}{\beta}, t+\delta n^2]$, $X_v(t') \notin (p-\varepsilon/8, 1-(p-\varepsilon/8))D_v(t')$. Then we have,
$$\P(A_t, t+\delta n^2< \tau'_0\mid \cf_t, t<\tau'_0, N_*(t)=pn)\leq e^{-\varepsilon_{20}n/2}.$$
\end{lemma}
\begin{proof}
This follows from a union bound and the previous lemma.
\end{proof}

\begin{lemma}
\label{l:degmartingale}
We have $\P[D_v(t+\delta n^2)\notin (\varepsilon/2, 1-\varepsilon/2)n, t+\delta n^2<\tau'_0\mid \cf_t, t<\tau_0]\leq e^{-\varepsilon_{21}n}$.
\end{lemma}

\begin{proof}
Let $C_t$ denote the event that for all $t'\in [t+1,t+\delta n^2]$, $|N_*(t')-N_*(t)|\leq \varepsilon n/16$. Let $H_t=A_t^{c}\cap C_t \cap \{t+\delta n^2 <\tau'_0\}$. Notice that on $H_t$, we have at time $t'\in [t,t+\delta n^2]$, the number of disagreeing edge at time $t$, $Z({t'})\geq (p(1-p)-\varepsilon/8)n^2$. Also notice that on $A_t$, the number of disagreeing edges incident to $v$ is at time $t'$ is in $[p(1-\varepsilon/8),1-p(1-\varepsilon/8)]D_v(t)$. Set $X(t')=D_v(t')-(1-3\varepsilon/4)n$, also set $\Delta^*=(1-\frac{\beta}{n})$. It follows therefore that for $\lambda>0$
\begin{eqnarray}
\label{e:degmartcalc}
\E(e^{\lambda X_{t'+1}}1_{H_t}\mid \cf_{t'})&\leq & e^{\lambda X_{t'}}\left(1+(e^{\lambda}-1)\frac{\Delta^* Z(t')}{N(n-1)}+(e^{-\lambda}-1)\frac{\Delta^*(p-\varepsilon/8)(X_{t'}+(1-3\varepsilon/4)n)}{2N}\right)\nonumber
\end{eqnarray}
 Now take $\lambda$ so small such that $e^{\lambda}-1\leq (1+\varepsilon/100)\lambda$ and $e^{-\lambda}-1\leq -(1-\varepsilon/100)\lambda$. Then we have

\begin{eqnarray}
\label{e:degmartcalc}
\E(e^{\lambda X_{t'+1}}1_{H_t}\mid \cf_{t'})&\leq & e^{\lambda X_{t'}}\left(1+ \lambda\Delta^*\left(\frac{(1+\varepsilon/50)Z_{t'}}{nN}-\frac{(p-\varepsilon/6) X_{t'}}{2N}-\frac{(p-\varepsilon/6)(1-3\varepsilon/4)n}{2N}\right) \right)\nonumber\\
& \leq & e^{\lambda X_{t'}}\left(1-\lambda\frac{\varepsilon X_{t'}}{8N}\right)\nonumber\\
&\leq & e^{\lambda_{*}X_{t'}}
\end{eqnarray}
where $0<\lambda_{*}<\lambda(1-\frac{\varepsilon}{10N})$ and since $p>\varepsilon$ implies that on $H_t$ we have $\frac{(1+\varepsilon/50)Z_{t'}}{nN}<\frac{(p-\varepsilon/6)(1-3\varepsilon/4)n}{2N}$. It follows that there exist $\lambda_0$ bounded away from $\lambda$ such that

$$\E(e^{\lambda X_{t+\delta n^2}}1_{H_t}\mid \cf_{t+Cn^2/\beta})\leq e^{\lambda_0 X_{t+Cn^2/\beta}}1_{H_t}\leq e^{\lambda_0\varepsilon n/4},$$

$$\text{i.e.},~\E(e^{\lambda X_{t+\delta n^2}}1_{H_t}\mid \cf_{t}, t<\tau_0)\leq e^{\lambda_0\varepsilon n/4}.$$ By Markov's inequality it now follows that

$$\P[D_v(t+\delta n^2)\geq (1-\varepsilon/2)n, H_t\mid \cf_t, t<\tau_0]\leq e^{-(\lambda-\lambda_0)\varepsilon n/4}.$$ It follows from Lemma \ref{l:deginterval} and another random walk estimate that $\P(H_t^c, t+\delta n^2<\tau'_0)$ is exponentially small in $n$, which completes the proof of one side of the bound in this lemma.

The other side of the bound can be proved similarly by starting with $\lambda$ negative and $X(t)=D_v(t)-3\varepsilon/4$. This completes the proof of the lemma.
\end{proof}

\begin{proof}[Proof of Theorem \ref{t:degsbound}]
This theorem follows from Lemma \ref{l:degmartingale} by taking a union bound over all vertices $v$, and all times $t'\in [t-\delta n^2, t]$ as in the proof of Theorem \ref{t:bigcut}, and using Theorem \ref{t:weakbound}.
\end{proof}

\subsection{Multiple-Edge Estimates}

\begin{theorem}
\label{t:multedgesbound}
We have for all $t\geq \delta n^2$, $\P[t+\delta n^2\geq \tau_4, t+\delta n^2< \tau_*, t<\tau]\leq \frac{1}{n^{4}}$.
\end{theorem}

The above theorem follows from the following lemma.

\begin{lemma}
\label{l:kedge}
Condition on $\{\cf_t, t<\tau_0\}$. Let $v$ be a vertex in $V$. Let $X_{v,k}(t)$ denote the number of vertices $u\in G$ such that $M_{uv}(t)=k$. Then for each $1\leq k\leq 2\varepsilon_{4}\log n$, we have
$$\P[X_{v,k}(t+\delta n^2)> C_110^{-k}n, t+\delta n^2<\tau'_0\mid \cf _t, t<\tau]\leq \frac{1}{n^{10}}.$$
\end{lemma}

The proof of the above lemma depends upon the next two lemmas.

\begin{lemma}
\label{l:kedge1}
Fix $v$ and $k$ as in the above lemma. Let $Y_{k,r}$ denote the number of $r$-edges that became $k$ edges in time $[t,t+\delta n^2]$. Fix an integer  $L_0\geq 1000C_2\delta \vee 1$. Then we have for each $k>3L_0$ and each $r<k-L_0$ we have
$\P(Y_{k,r}\geq 3^{-(k-r)}C10^{-k}n, t+\delta n^2<\tau'_0\mid \cf_t, t<\tau_0)\leq \frac{1}{n^{20}}$.
\end{lemma}

\begin{proof}
This proof is similar to the proof of Lemma \ref{l:lbw4} and we give only a sketch. Condition on $\{\cf_t, t<\tau_0\}$. Let $u_1,u_2,\ldots, u_{D}$ be the vertices in $V$ such that $\{M_{vu_i}(t)=r\}$. Now without loss of generality we assume $D=C10^{-r}n$. Let $T_i$ be the number of times a rewired edges rewires along $vu_i$. It is clear that on $\{t+\delta n^2<\tau'_0\}$, $T_i$ is stochastically dominated by a $\mbox{Bin}(\delta n^2, \frac{10C_2}{n^2})$. Let $Z_i$ denote the indicator that $T_i\geq (k-r)$. Using a joint stochastic domination argument as in the proof of Lemma \ref{l:lbw4} it is easy to show that upto terms much smaller than $n^{-100}$, $\P[\sum_{i}Z_{i}\geq (30)^{-(k-r)}D]$ can be approximated by $\P[\sum_{i}Z'_{i}\geq (30)^{-(k-r)}D]$ where $Z'_i$ are i.i.d. $\mbox{Bin}(\delta n^2, \frac{20C_2}{n^2})$. For our choice of $L_0$ it follows that for all $\ell\geq L_0$, $\P(\mbox{Bin}(\delta n^2, \frac{10C_2}{n^2})\geq \ell)\leq 40^{-\ell}$. Now using another Chernoff bound gives us that on $\{t+\delta n^2<\tau'_0\}$
$$P[Y_{k,r}\geq (30)^{-(k-r)}C_110^{-r}n]\leq \frac{1}{n^{20}}.$$  
%
%
%
This completes the proof of the lemma.
\end{proof}

\begin{lemma}
\label{l:kedge2}
Let $L_0$ be chosen as in Lemma \ref{l:kedge1} and $k>\frac{10^6L_0}{\varepsilon\delta}$. Let $k-L_0\leq r\leq k$ be fixed. Let $u_1,u_2,\ldots, u_{D}$ be the vertices in $G$ such that $\{M_{vu_i}(t)=r\}$. Let $R_i$ be the number of edges lost by $vu_i$ in time $[t,t+\delta n^2]$. Let $Z'_i=1_{\{R_i\leq L_0\}}$. Then
$$\P[\sum_{i=1}^{D}Z'_i\geq 20^{-L_0}C10^{-r}n]\leq \frac{1}{n^{20}}.$$
\end{lemma}

For the proof of this lemma we shall consider the continuous time \emph{rewire-to-random-*} dynamics. We need the following proposition.

\begin{proposition}
\label{p:kedge2continuousdisagreeing}
Let us condition on $\{\cf_t, t<\tau_0\}$. Let $v$ be a fixed vertex in $V$. Let $u_1,u_2,\ldots , u_{D}$ be the vertices in $v$ such that we have $M_{vu_i}(t)=r$, where $1<r<2\varepsilon_4\log n$. Set $G(t)=H(0)$ and Run the following continuous time \emph{rewire-to-random-*} process $H(\cdot)$ from time $0$ to $\delta/2$. Each directed edge rings at rate $1$. If the endpoints of the edge are agreeing in the current graph, no change occurs. If they are diasgreeing then we do a voter model step with probability $\frac{\beta}{n}$ and a rewire-to-random step with probability $(1-\frac{\beta}{n})$. Let $Z_i$ be the indicator that $(v,u_i)$ is disagreeing for less that $\frac{\varepsilon\delta}{200}$ time in $H(\cdot)$. Then we have $P(\sum_{i=1}^{D} Z_i > 25^{-L_0}D, \tau'_0>\delta/2 )\leq e^{-\gamma\sqrt{n}}$ for some $\gamma>0$.
\end{proposition}

\begin{proof}
Without loss of generality, we assume $D=C_110^{-r}n >> \sqrt{n}$. Let us choose a random subset $D^*\subseteq {1,2,\ldots, D}$ with $|D^*|=\sqrt{n}$. It therefore suffices to prove that $\P(\sum_{i\in D^*} Z_i> 30^{-L_0}\sqrt{n}\mid D^*)\leq e^{-\gamma\sqrt{n}}$. This fact is established by Lemma \ref{l:continuousdisagreeingkedgesample} and Lemma \ref{l:kedgecontdependence} below which completes the proof.
\end{proof}

\begin{lemma}
\label{l:continuousdisagreeingkedgesample}
Condition on  $\{\cf_t, t<\tau_0\}$. Let us set $G(t)=H(0)$, and consider running the continuous time \emph{rewire-to-random-*} dynamics $H(\cdot)$ from time $0$ to $\delta/2$. Let $v, v_1,v_2,\ldots , v_{\sqrt{n}}$ be fixed vertices in $V$. Let us consider the independent continuous time random walks described in \S~\ref{s:coupling}. Let us $X_i^{j}(\cdot)$ be the random walk started from $v_i$ on $H(2j\frac{C}{\beta})$ run for time $\frac{2C}{\beta}$. Let for $\frac{C}{\beta}\leq s \leq \frac{2C}{\beta}$,  $Y_i(2j\frac{C}{\beta}+s)$ is the indicator of the event that the opinion of $X_i^{j}(s)$ in $H(0)$ is different from the opinion of $v$ in $H(s)$. Let $Y_i^*=\int_{0}^{\frac{\delta}{2}}Y_i(s)~ds$. Further, let $Z_i^*$ denote the indicator that $Y_i^*< \frac{\varepsilon\delta}{64}$. Then we have,
$$\P\left[\sum_{i} Z_i^* \geq 40^{-L_0}\sqrt{n}\right]\leq e^{-c\sqrt{n}}.$$
\end{lemma}

\begin{proof}
For $k=0,2,\ldots , \frac{C}{\theta}-1$, $j=1,2,\ldots \frac{\delta \beta}{4C}$, let $\chi_v^{j,k}=1$ if $v$ spends majority of its time in the interval $[(2j-1)\frac{C}{\beta}+\frac{k\theta}{\beta},(2j-1)\frac{C}{\beta}+\frac{(k+1)\theta}{\beta}]$ with opinion $1$ and $0$ otherwise. Let $\chi_i^{j,k}=1$ if the opinion of $X_i^{j}(s)=1$ for all $s\in [(2j-1)\frac{C}{\beta}+\frac{k\theta}{\beta},(2j-1)\frac{C}{\beta}+\frac{(k+1)\theta}{\beta}]$, $\chi_i^{j,k}=0$ if the opinion of $X_i^{j}(s)=0$ for all $s\in [(2j-1)\frac{C}{\beta}+\frac{k\theta}{\beta},(2j-1)\frac{C}{\beta}+\frac{(k+1)\theta}{\beta}]$, and $\chi_i^{j,k}=\star$ otherwise. Now let $U_i^{j,k}=1_{\{\chi_i^{j,k}=1, \chi_v^{j,k}=0\}}+1_{\{\chi_i^{j,k}=1, \chi_v^{j,k}=0\}}$. Let us fix $k$. Now choose $\theta$ sufficiently small so that the chance that the random walk takes a step in time $\theta/\beta$ is at most $\frac{\varepsilon}{4}$. Clearly, for a fixed realisation of the sequence $\chi_v^{j,k}$, and on $\{2(j-1)C/\beta<\tau'_0\}$, we have by Lemma \ref{l:couplingtv}, that  $\E[U_i^{j,k}\mid \cf_{2(j-1)C/\beta}]\geq \varepsilon/4$. Since the random walks are independent, it follows that

$$\P\left[\#\{i: \sum_j U_i^{j,k}\leq \frac{\delta \beta \varepsilon}{8C}\}\geq e^{-\gamma'\beta}\sqrt{n}, \chi_v^{j,k}, \frac{\delta}{2}<\tau'_0\right]\leq e^{-c\sqrt{n}}.$$

Taking a union bound over $2^{\delta\beta/4C}$ possible realisations of the sequence $\chi_v^{k,j}$ (for a fixed $k$), we get that,

$$\P\left[\#\{i: \sum_j U_i^{j,k}\leq \frac{\delta \beta \varepsilon}{32C}\}\geq e^{-\gamma'\beta}\sqrt{n}, \frac{\delta}{2}<\tau'_0\right]\leq e^{-c\sqrt{n}}$$
for some constant $\gamma'$ and $c>0$.

Now taking a union bound over $k$ we get,

$$\P\left[\#\{i: \sum_j\sum_{k} U_i^{j,k}\leq \frac{\delta \beta \varepsilon}{32\theta}\}\geq \frac{C}{\theta}e^{-\gamma'\beta}\sqrt{n}, \frac{\delta}{2}<\tau'_0\right]\leq \frac{C}{\theta}e^{-c\sqrt{n}}.$$
Now notice that on $\{\sum_j\sum_{k} U_i^{j,k}>\frac{\delta \beta \varepsilon}{32\theta}\}$, we have $Y_i^*>\frac{\varepsilon\delta}{64}$, and the proof of the lemma is completed by taking $\beta$ sufficiently large.
\end{proof}

\begin{lemma}
\label{l:kedgecontdependence}
Assume the setting of Lemma \ref{l:continuousdisagreeingkedgesample}. Also let $\tilde{Y}_i$ denote the amount of time the bond $(v,v_i)$ is disagreeing in $[0,\frac{\delta}{2}]$. Then there is a coupling of the continuous time evolving voter model with the continuous time random walks started at $v_i$ as described in Lemma \ref{l:continuousdisagreeingkedgesample} such that
$$\P[\#\{i: \tilde{Y}_{i}\leq Y_i^*- \frac{\varepsilon\delta}{128}\}\geq 40^{-L_0}\sqrt{n}, \frac{\delta}{2}<\tau'_0]\leq e^{-c\sqrt{n}}$$
for some constant $c$.
\end{lemma}

\begin{proof}
Consider the coupling described in \S~\ref{s:coupling},
with the obvious modification for the continuous time dynamics. Define $Y_{i,j}=0$ if the opinion of $X_i^j(s)$ is the same as the opinion of $v_i$ in $H(\frac{2jC}{\beta}+s)$ for all $s\in [0, \frac{2C}{\beta}]$ and $Y_{i,j}=1$ otherwise. It follows by arguments similar to those in the proof of Lemma \ref{l:couplemany} that
$$\P\left[\sum_{j=1}^{\delta\beta/4C}\sum_{i=1}^{\sqrt{n}}Y_{i,j} \geq 10C\delta\sqrt{\beta}\sqrt{n}\right]\leq e^{-c\sqrt{n}}$$.

Again notice that, $\tilde{Y}_i-Y_i^* \geq (\sum_{j}Y_{i,j})\frac{2C}{\beta}$. It follows that,

$$\P[\#\{i: \tilde{Y}_{i}\leq Y_i^*- \frac{\varepsilon\delta}{64}\}\geq \frac{2560C^2}{\sqrt{\beta}\varepsilon\delta}\sqrt{n}]\leq e^{-c\sqrt{n}}$$
The proof of the lemma is completed by taking $\beta$ sufficiently large.
\end{proof}

\begin{lemma}
\label{l:kedge2continuous}
Condition on $\{\cf_t, t<\tau_0\}$. Assume the setting of Proposition \ref{p:kedge2continuousdisagreeing}. Then, Let $W_i^*$ be the indicator that the bond $vu_i$ loses at least $L_0$ edges by time $\frac{\delta}{2}$. For $r> \frac{10^6L_0}{\varepsilon\delta}$, we have
$$\P\left[\sum_{i=1}^{D}W^*_i\geq 20^{-L_0}C_110^{-r}n, \tau'_0>\frac{\delta}{2}\right]\leq \frac{1}{n^{20}}.$$
\end{lemma}

\begin{proof}
Without loss of generality assume $D=C_110^{-r}n>> \sqrt{n}$. In the continuous time model, the rate at which a bond with $k$-disagreeing edges lose an edge is $k$, and the rings of the different edges are independent. Let $S_1, S_2,\ldots, S_{D}$ be the number of edges lost by the bonds $vu_1, vu_2,\ldots vu_D$ respectively. Let $S'_i$ be independent $\mbox{Poi}(1000L_0)$ variables. It follows from Proposition \ref{p:kedge2continuousdisagreeing} that there exist a coupling such that,
$$\P[\#\{i: S_i\leq S'_i\wedge L_0 \}\geq 25^{-L_0}C_1 10^{-k}n, \tau'_0>\frac{\delta}{2}]\leq e^{-c\sqrt{n}}$$ for some constant $c>0$.
Also notice that

$P[\#\{i: S'_i\leq L_0\}\geq 40^{-L_0}D]\leq e^{-c\sqrt{n}}$. The lemma follows.
\end{proof}

Now we are ready to prove Lemma \ref{l:kedge2}.

\begin{proof}[Proof of Lemma \ref{l:kedge2}]
The proof follows from the obvious coupling of the continuous time {rewire-to-random-*} dynamics with the discrete time {rewire-to-random-*} dynamics and observing that with exponentially high probability, the number of step taken in the discrete time process upto time $\frac{\delta}{2}$ in the continuous time process is less that $\delta n^2$.
\end{proof}

We now complete the proof of Lemma \ref{l:kedge}.

\begin{proof}[Proof of Lemma \ref{l:kedge}]
Choose $C_1$ large enough such that the condition is satisfied for all $k<\frac{10^6L_0}{\varepsilon\delta}$ by the degree estimate. Now notice that $$X_{v,k}(t+\delta n^2)\leq \sum_{r\geq k+1} X_{v,r}(t)+\sum_{r=0}^{k-L_0}Y_{r,k}+\sum_{r=k-L_0}^{k} Y_{r,k}.$$
On $\{t<\tau_0\}$ the first sum is $\frac{C_110^{-k}n}{9}$. By Lemma \ref{l:kedge1} we have that
$$\P\left[\sum_{r=0}^{k-L_0}Y_{r,k}\geq 3^{-L_0}C10^{-k}n, t+\delta n^2<\tau'_0\right]\leq \frac{1}{n^{19}}.$$
Also using Lemma \ref{l:kedge1} and Lemma \ref{l:kedge2} it follows that
$$\P\left[\sum_{r=k-L_0}^{k}Y_{r,k}\geq 2^{-L_0}C10^{-k}n, t+\delta n^2<\tau'_0\right]\leq \frac{1}{n^{19}}.$$
Putting together all these gives us the statement of the lemma.
\end{proof}

\begin{proof}[Proof of Theroem \ref{t:multedgesbound}]
For each fixed $v$, taking a union bound over all $k$, and then taking a union bound over all $v$, and then taking a union bound over times in $[t-\delta n^2, t]$ yields  the theorem from Lemma \ref{l:kedge}.
\end{proof}

\subsection{Completing the proof of Theorem \ref{t:stoptime}}

Now we are ready to prove Theorem \ref{t:stoptime}.

\begin{proof}[Proof of Theorem \ref{t:stoptime}]
Using a random walk estimate it is clear that $\P[\tau_0>n^4]=o(1)$. Also it is clear from the properties of an Erd\H{o}s-R\'enyi graph that $P[\tau_0<\delta n^2]=o(1)$. Now for $k\geq 0$, and $i=2,3,4,5$, let $A_{k,i}$ denote the event $\{k\delta n^2<\tau_0, (k+1)\delta n^2<\tau_*, (k+1)\delta n^2\geq \tau_i\}$. Using Theorem \ref{t:bigcut}, Theorem \ref{t:indmultbound}, Theorem \ref{t:multedgesbound} and Theorem \ref{t:degsbound} and taking a union over $0\leq k \leq n^2/\delta$, it follows that $$\P[\tau_0<\tau_*-\delta n^2]\leq o(1)+\sum_{i,k} \P(A_{k,i}) =o(1).$$
This completes the proof.
\end{proof}

\section{Rewire-to-random Eventually Splits}
\label{s:esplit}

In this section we prove Theorem \ref{t:rtoresplit}. For this section we shall consider running the \emph{rewire-to-random} model with a different initial condition. For $0<p<1$, let $\mathcal{G}^{*}(p)$ be the subset of the state space of our markov chain, i.e., let $\mathcal{G}^*(p)$ is a set of multi-graphs of $n$ vertices with labelled edges where each vertex has either of the two opinions $0$ and $1$, such that $N_1(G)=pn$ and the number of edges in $G$ is in $[\frac{12n^2}{50},\frac{13n^2}{50}]$.
Theorem \ref{t:rtoresplit} will follow from the the next theorem.

\begin{theorem}
\label{t:es}
Let $\beta>0$ be fixed. Consider running the \emph{rewire-to-random} model with relabelling rate $\beta$ starting with the state $G(0)$. Consider the stopping times $\tau=\min\{t:\mathcal{E}^{\times}(t)=\emptyset\}$ and $\tau_*=\tau_{*}(p/2)=\min\{t:N_*(t)\leq \frac{pn}{2}\}$. Then there exists $p=p(\beta)$ sufficiently small such that for all $G(0)\in \mathcal{G}^{*}(p)$, we have $\tau<\tau_{*}$ with high probability.
\end{theorem}

Before starting with the proof of Theorem \ref{t:es} we make the following definitions. Let us fix $G(0)\in \mathcal{G}^{*}(p)$. Let $S$ be the set of vertices in $G(0)$ with degree at most $10n$ and let $T$ be the set of vertices with degree more than $10n$. Clearly $|S|\geq \frac{24n}{25}$. Let us run the process till $10n^2$ steps.

Let $W_{SS}$ denotes the total number of rewirings of edges with both endpoints in $S$. $W_{ST}$ and $W_{TT}$ are defined similarly. Let $Y_{SS}$  denote the number of edges with both endpoints in $S$ at the end of the process (i.e., after running $10n^2$ steps). $Y_{ST}$, $Y_{TT}$ are defined similarly.

We next describe an equivalent way of  constructing the \emph{rewire-to-random} dynamics.

\subsection*{An Equivalent Construction of the Dynamics}
Let $\{X_i\}$ and $\{X'_i\}$ be two sequences of i.i.d. $\mbox{Geom}(\frac{\beta}{n})$ variables (taking values in $\{0,1, \ldots\}$). Let $\{Z_i\}$ be a sequence of i.i.d. $\mbox{Ber}(\frac{1}{2})$ variables. Let $\{W_i\}$ be a sequence of vertices of $G$ where each $W_i$ is uniformly chosen vertex of $G$. All these sequences are distributed independently of each other.

We now describe how to run the process starting with $G(0)$ using only the randomness in the above sequences and the randomness used to choose a disagreeing edge uniformly at each step. Having chosen a disagreeing edge the variables $Z_i$ will be used to designate one of the endpoints of the edge uniformly as the root of the current (relabelling or rewiring) update.  
Also for each vertex $v$ in $V$ we shall define a sequence $K_i(v)$.  To start with, list the vertices in $V$ in some order, say $\{v_1,v_2,\ldots, v_n\}$. Define $K_0(v_j)=X'_{j}$ for all $j$. This encodes the number of updates at the vertex $v_j$  (i.e., the number of moves with $v_j$ being the root) before it changes its opinion for the first time which clearly has a $\mbox{Geom}(\frac{\beta}{n})$ distribution. Roughly speaking, for each vertex $v$ the sequence $\{K_i(v)\}$ will be a counter which shall denote how many more rewiring updates one needs to make at $v$ before the next relabelling update. Once the counter runs to $0$, the next update at that vertex is a relabelling one, and a new value from either the sequence $\{X_i\}$ or the sequence $\{X'_i\}$ will be assigned to the counter. We describe the process formally below.   

We shall define the sequences $L_{i}, L'_{i}, T_{i}$ recursively, these will be indices of different elements chosen from $\{X_i\}$, $\{X'_i\}$ and $\{W_i\}$ respectively. Let $L_0=T_0=0$ and $L'_{0}=n$. At step $i$, pick a disagreeing edge $e$ uniformly at random, if such an edge exists. If $Z_i=1$, then choose the vertex with opinion $1$ to be the root of the rewiring or relabelling step, if $Z_i=0$, choose the other one. Let $v$ be the chosen vertex. 
If $v$ is in $S$ and the opinion of $v$ is $0$, do the following.
Set $L'_{i}=L'_{i-1}$. If $K_{i-1}(v)$ is positive, then define $T_{i}=\min\{k>T_{i-1}:W_{k}\neq v\}$, that is $T_i$ is the index of the first hitherto uninspected element in $\{W_j\}$ which allows a legal rewiring move. Rewire the edge to $W_k$ and reduce $K_{i}(v)$ by 1, and set $L_{i}=L_{i-1}$. If $K_{i-1}(v)=0$, then relabel $v$ and set $L_i=L_{i-1}+1$ and $K_i(v)=X_{L{i}}$, in this case, also set $T_{i}=T_{i-1}$. If $v$ is not in $S$, or the opinion of $v$ is $1$, then do the same as in the previous case except use elements from the sequence $X'$ and $L'$ in stead of the elements from sequence $X$ and $L$, and change the values in the sequence $L'$ instead of the sequence $L$.

It is easy to see that this is indeed an implementation of the \emph{rewire-to-random} dynamics.

We need the following lemmas.  The first follows immediately from the fact that the number of vertices with opinion $1$ does a random walk.

\begin{lemma}
\label{l:es1}
For a fixed $p>0$ and $G(0)\in \mathcal{G}^*(p)$, the number of vertices of opinion $1$ remains between $pn/2$ and $3pn/2$ throughout the first $10n^2$ steps w.h.p..
\end{lemma}

We call an element of the sequence $X$ \emph{stubborn} if it is at least $25n$.

\begin{lemma}
\label{l:es2}
Let $Y=\#\{i\leq L_{10n^2}: X_{i} > 25 n\}$ denote the number of stubborn elements $X$ which are used in first $10n^2$ steps. Then with high probability, $N_1(10n^2)\geq Y$, i.e., the number of vertices with label $1$ after $10n^2$ steps is at least the number of used \emph{stubborn} elements of the sequence $X$.
\end{lemma}

\begin{proof}
Let $\mathcal{S}$ denote the following event.
$$\mathcal{S}=\Bigl\{\forall v\in G: \#\{i\leq T_{10n^2}:W_i=v\}\leq 14n\Bigr\}.$$
We show that on $\mathcal{S}$, the vertices in $S$ that each stubborn element gets assigned to (i.e., those $v$ such that $X_{\ell}=K_i(v)$ for some $i$ and some stubborn $X_{\ell}$) are distinct and each of them has label $1$ after $10n^2$ steps. Consider a specific stubborn element, suppose it was used and assigned to the vertex $v$. By definition, at that point the opinion of $v$ was $1$. Now by definition of \emph{stubbornness}, before it changes its opinion again, $v$ needs to be the root of at least $25n$ rewiring moves. Notice now that the number of rewirings rooted at $v$ is at most the sum of the initial degree of $v$ (which is at most $10n$ since $v\in S$) and the number of rewirings to $v$ (which is at most $14n$ on $\mathcal{S}$). Hence the vertex $v$ never changes its opinion again and in particular is never associated with any other stubborn element. Hence corresponding to each used stubborn element, there are distinct vertices in $V$ which have opinion $1$ after $10n^2$ steps.

It remains to show that $\mathcal{S}$ occurs with high probability. First notice that using an argument similar to that used in the proof of Lemma \ref{l:betasmallincoming}, it follows that $\P(T_{10n^2}>11n^2)$ is exponentially small in $n$. Also, we note that for each $v\in V$, the chance that $v$ occurs more than $14n$ times in the first $11n^2$ elements of the list $W$ is exponentially small in $n$ using a Chernoff bound. Taking a union bound over all the vertices completes the proof of the lemma.
\end{proof}

\begin{lemma}
\label{l:es3}
Let $RL_{SS}$ denote the number of times a relabelling occurs when an edge with both endpoints in $S$ was chosen. Then, for $p$ sufficiently small, $RL_{SS}\leq \frac{\beta n}{20}$ w.h.p..
\end{lemma}

\begin{proof}  Let $RL_{SS}^+$ denote the number of times we have a relabelling changing an opinion to $1$ after an edge with both endpoints in $S$ was chosen. Notice that each time we choose an edge with both endpoints in $S$, and do a relabelling update changing an opinion from $0$ to $1$, and element from the sequence $X$ gets used. Hence it follows from Lemma \ref{l:es1} and Lemma \ref{l:es2} it follows that w.h.p.
$$RL_{SS}^{+}\leq \min_i\Bigl\{\#\{j\leq i: X_j\geq  25n\}>\frac{3pn}{2} \Bigr\}.$$ 

Since each $X_j$ is a $\mbox{Geom}(\frac{\beta}{n})$ variable it follows that $\P(X_j\geq 25n)=(1-\frac{\beta}{n})^{25n}\geq e^{-50\beta}$. Since $X_j$'s are i.i.d., it follows that for $p$ sufficiently small (depending on $\beta$), within first $\frac{\beta n}{50}$ elements of $X$, there are more than $\frac{3pn}{2}$ \emph{stubborn} elements with high probability. It follows that with high probability $RL_{SS}^{+}\leq \frac{\beta n}{50}$.

Now notice that each time a relabelling occurs when an edge with both endpoints in $S$ is chosen, with probability $\frac{1}{2}$ the relabelling changes an opinion to $1$ and these events are independent of one another. It follows that $\P(RL_{SS}>\frac{\beta n}{20}, RL_{SS}^{+}\leq \frac{\beta n}{50})$ is exponentially small in $n$. It now follows that for $p$ sufficiently small, $RL_{SS}\leq \frac{\beta n}{20}$ w.h.p..
\end{proof}

\begin{lemma}
\label{l:es4}
Let $R_{SS}$ be the number of times an edge with both endpoints on $S$ was picked. For $p$ sufficiently small, $R_{SS}\leq \frac{n^2}{10}$ w.h.p..
\end{lemma}
\begin{proof}
Each time an edge is picked, it leads to a relabelling with probability $\frac{\beta}{n}$. It follows using a Chernoff bound that
$$\P(R_{SS} > \frac{n^2}{10}, RL_{SS}\leq \frac{\beta n}{20})\leq \P(\mbox{Bin}(\frac{n^2}{10}, \frac{\beta}{n})\leq \frac{\beta n}{20})\leq e^{-\frac{\beta n}{80}}.$$
The lemma now follows from Lemma \ref{l:es3}.
\end{proof}

\begin{lemma}
\label{l:es5}
Let $R_{ST}$ be the number of times a disagreeing edge is picked with one endpoint in $S$ and another in $T$. For $p$ sufficiently small $R_{ST}\leq 3n^2$ w.h.p..
\end{lemma}

\begin{proof}
Let $W_{ST}$ denote the number of rewiring moves where a disagreeing edge with one endpoint in $S$ and another endpoint in $T$ is rewired. Each time an edge with one endpoint in $S$ and the other in $T$ is rewired, with probability $\frac{1}{2}$ it is rewired with the root at the vertex in $S$ and independent of that with probability at least $\frac{|S|-1}{n-1}$ the edge is rewired to a vertex in $S$. That is, after rewiring an edge with one endpoint in $S$ and another in $T$, the chance that it becomes an edge with both endpoints in $S$ is at least $\frac{|S|-1}{2(n-1)}\geq \frac{23}{50}$ for $n$ sufficiently large. Let $W_{ST\rightarrow SS}$ denote the number of such rewirings. It follows that
$\P(W_{ST\rightarrow SS}\leq \frac{11n^2}{10}, W_{ST}> \frac{5n^2}{2})$ is exponentially small in $n^2$. Now notice that,
$$n^2\geq Y_{SS}\geq -R_{SS}+W_{ST\rightarrow SS}$$
and it follows from Lemma \ref{l:es4} that for $p$ sufficiently small $W_{ST\rightarrow SS}\leq \frac{11n^2}{10}$ w.h.p. and hence $W_{ST}\leq \frac{5n^2}{2}$ w.h.p.. Since each time a disagreeing edge is picked, with probability $1-\frac{\beta}{n}$, it leads to a rewiring, it follows that $\P(R_{ST}>3n^2, W_{ST}\leq \frac{5n^2}{2})$ is exponentially small in $n^2$ and hence for $p$ sufficiently small $R_{ST}\leq 3n^2$ w.h.p..
\end{proof}

\begin{lemma}
\label{l:es6}
Let $R_{TT}$ be the number of times a disagreeing edge is picked with both endpoints in $T$. For $p$ sufficiently small $R_{TT}\leq 6n^2$ w.h.p..
\end{lemma}

\begin{proof}
Arguing as in the proof of Lemma \ref{l:es5} we have that after rewiring an edge with both endpoints in $T$, the chance that it becomes an edge with one endpoint in $S$ and another in $T$ is $\frac{|S|}{n}\geq \frac{24}{25}$. Let $W_{TT\rightarrow ST}$ denote the number of such rewirings. It follows that $\P(W_{TT\rightarrow ST}\leq 4n^2, W_{TT}> 5n^2)$ is exponentially small in $n^2$. Now notice that
$$n^2\geq X_{ST}\geq -R_{ST}+W_{TT\rightarrow ST}$$
and it follows from Lemma \ref{l:es5} that for $p$ sufficiently small $W_{TT\rightarrow ST}\leq 4n^2$ w.h.p.. It follows that $W_{TT}\leq 5n^2$ w.h.p.. Arguing as in the proof of Lemma \ref{l:es5} we conclude that $R_{TT}\leq 6n^2$ w.h.p..
\end{proof}

We are now ready to prove Theorem \ref{t:es}.

\begin{proof}[Proof of Theorem \ref{t:es}]
From Lemma \ref{l:es4}, Lemma \ref{l:es5} and Lemma \ref{l:es6} we have that for $p=p(\beta)$ sufficiently small and $G(0)\in \mathcal{G}^{*}(p)$, we have $R_{SS}+R_{ST}+R_{TT}< 10n^2$ with high probability. This implies after $10n^2$ steps there are no disagreeing edge in the graph, i.e., for $p$ sufficiently small $\tau\leq 10n^2$ w.h.p.. The theorem now follows from Lemma \ref{l:es1}.
\end{proof}

Now we prove Theorem \ref{t:rtoresplit} using Theorem \ref{t:es}.

\begin{proof}[Proof of Theorem \ref{t:rtoresplit}]
Since in the set up of this theorem $G(0)$ is distributed as $G(n,\frac{1}{2})$, it follows that with high probability the number of edges in $G(0)$ is in $[\frac{12n^2}{50},\frac{13n^2}{50}]$. Let $p=p(\beta)$ be sufficiently small so that the conclusion of Theorem \ref{t:es} holds. It follows that on $\{\tau_{*}(p)<\infty\}$, with high probability $G(\tau_{*}(p))\in \mathcal{G}^{*}(p)$. Since the \emph{rewire-to-random} dynamics is symmetric in the opinions $0$ and $1$, it follows from Theorem \ref{t:es} that $\tau<\tau_{*}(\frac{p}{2})$ with high probability. This completes the proof of the theorem.
\end{proof}


\section{Modifications for the Rewire-to-same Model}
\label{s:samemod}
We can prove  Theorem \ref{t:rewiretorandom} for the rewire-to-same model in a similar manner. For the sake of not being repetitive, we only point out the main differences here. Notice that, with respect to our proof in previous sections, the major difference between the two dynamics that causes some inconvenience is that in the \emph{rewire-to-same} dynamics a vertex with minority opinion is likely to receive edges at a higher rate than a vertex with majority opinion. But as long as the minority opinion density does not become too small, the difference is of a bounded factor, and it turns out that the arguments can be modified to accommodate this. We now point out the main lemmas from the previous sections that need to be modified for the \emph{rewire-to-same} dynamics.

\subsection{Small $\beta$ Case:}
Notice that the only place that needs a modification is Lemma \ref{l:betasmallincoming}. One needs to define for this case $L_{i+1}$ as the first entry after $L_{i}$ to which a rewiring move is legal. Here it is not true that $L_{6n^2}\leq 13n^2/2$ with exponentially high probability. But notice that, on $t<\tau_*(1/3)$, $L_{i+1}-L_{i}$ is stochastically dominated by a $\mbox{Geom}(\frac{1}{3})$ variable and hence, one can say $L_{6n^2}<20n^2$ with exponentially large probability. The rest of the proof is in essence same up to some minor changes in constants.

\subsection{Large $\beta$ Case:}
The proof in the large $\beta$ case also follows along the similar lines. Instead of the \emph{rewire-to-same} dynamics, we consider the \emph{rewire-to-same-*} dynamics, where instead of a disagreeing edge, at each step we pick an edge at random, and do not do anything if the edge happens to be agreeing. Most of necessary changes occur while bounding the number of incoming edges to a vertex in time $[t,t+\delta n^2]$. But on $\tau<\tau_*(\varepsilon)$, one can bound the number of incoming edges to a vertex $v$ in that time by a $\mbox{Bin}(\delta n^2, \frac{1}{\varepsilon n})$ variable. 

The only other place where a somewhat significant modification is necessary is in the bound for large cuts. Instead of Proposition \ref{p:proportionofbadedges} and Proposition \ref{p:proportionofbadedgesss}, one needs to show that for any given cut $S$ and $T$, with $|S|\wedge |T|\geq \varepsilon_2 n$, the number of edges with one endpoint in $S$ and another endpoint in $T$, such that the $S$ end point has opinion $0$ and the $T$ endpoint has opinion $1$, is roughly about $p(1-p)$ fraction of the total number of edges with exponentially high probability, and similar bounds on other similar quantities. It can be checked that all these can be obtained following a similar line of arguments as in the proof of Proposition \ref{p:proportionofbadedges}. We omit the details. The rest of the bounds can then be obtained by suitably modifying the martingale calculations in Proposition \ref{p:bigcut}. The whole proof can then be carried out with some minor changes of constants.

\section{Open Questions}
\label{s:conclusion}
While we establish some rigorous results for the evolving voter model on dense random graphs, the picture is far from clear. We conclude with the following natural questions, that are still open.

\begin{itemize}
\item {\bf What happens eventually in the \emph{rewire-to-same} model?} Notice that we do not have any result analogous to Theorem  \ref{t:rtoresplit} in the \emph{rewire-to-same} model. It is a natural question to ask whether in the rewire-to-same model, is there a positive fraction of both opinions present when the process reaches an absorbing state? As we have mentioned before in \cite{Durrett12} it was conjectured that, for the sparse graphs (with constant average degree), in the rewire-to-same model, one of the opinions take over almost the whole graph, but it is not known rigorously.

\item{\bf Is there a sharp transition in $\beta$?} Another natural question to ask is if there is any value $\beta_0$ such that for $\beta<\beta_0$ we have behaviour as in Theorem \ref{t:rewiretorandom}$(i)$ and for $\beta>\beta_0$, we have behaviour as in  Theorem \ref{t:rewiretorandom}$(ii)$?

\item{\bf What can we say about sparser graphs?}
We can prove by an argument similar to proof of Theorem \ref{t:rtoresplit}. for sparser graph (i.e., $G(0)\sim G(n,\frac{\lambda}{n})$) that there exists $\beta_0>0$, such that for all $\beta<\beta_0$, with high probability the process stops before the density of the opinions change. But the other side of the phase transition seems harder to prove. The main difficulty seems to be the presence of a few vertices of very high degree. It is another interesting question to investigate whether one could prove results about evolving voter models starting with not too sparse graphs, e.g. $G(0)\sim G(n, n^{1-\alpha})$ for some $0<\alpha <1$?

\end{itemize}

\medskip

\textbf{Acknowledgements.} The authors would like to thank Rick Durrett for introducing us to the model.

\bibliography{evm}

\begin{thebibliography}{10}

\bibitem{AF02}
David Aldous and James~Allen Fill.
\newblock Reversible markov chains and random walks on graphs, 2002.
\newblock Unfinished monograph, recompiled 2014, available at
  \url{http://www.stat.berkeley.edu/~aldous/RWG/book.html}.

\bibitem{Durrett12}
R.~Durrett, J.P. Gleeson, A.L. Lloyd, P.J. Mucha, F~Shi, D.~Sivakoff, J.E.S.
  Socolar, and C.~Varghese.
\newblock Graph fission in an evolving voter model.
\newblock {\em Proc. Natl. Acad. Sci. USA}, 109:3682--3687, 2012.

\bibitem{Dur06}
Rick Durrett.
\newblock {\em Random Graph Dynamics (Cambridge Series in Statistical and
  Probabilistic Mathematics)}.
\newblock Cambridge University Press, New York, NY, USA, 2006.

\bibitem{FN02}
Shmuel Friedland and Reinhard Nabben.
\newblock On cheeger-type inequalities for weighted graphs.
\newblock {\em Journal of Graph Theory}, 41(1):1--17, 2002.

\bibitem{Gil06}
Santiago Gil, Dami\'an~H. Zanette, Centro~At\'omico Bariloche, and R\'io Negro.
\newblock Coevolution of agents and networks: Opinion spreading and community
  disconnection.
\newblock {\em Phys. Lett. A}, 356:89--94, 2006.

\bibitem{GB08}
Thilo Gross and Bernd Blasius.
\newblock Adaptive coevolutionary networks: a review.
\newblock {\em Journal of The Royal Society Interface}, 5(20):259--271, 2008.

\bibitem{Henry11}
A.~D. Henry, P.~Pralat, and C-Q Zhang.
\newblock Emergence of segregation in evolving social networks.
\newblock {\em Proceedings of the National Academy of Sciences}, 108, 2011.

\bibitem{HL1975}
Richard~A. Holley and Thomas~M. Liggett.
\newblock Ergodic theorems for weakly interacting infinite systems and the
  voter model.
\newblock {\em Ann. Probab.}, 3(4):643--663, 1975.

\bibitem{HN06}
P.~Holme and M.~Newman.
\newblock Nonequilibrium phase transition in the coevolution of networks and
  opinions.
\newblock {\em Phys. Rev. E}, 74(056108), 2006.

\bibitem{PRE08KH}
Daichi Kimura and Yoshinori Hayakawa.
\newblock Coevolutionary networks with homophily and heterophily.
\newblock {\em Phys. Rev. E}, 78, Jul 2008.

\bibitem{KB08}
B.~Kozma and A.~Barrat.
\newblock Consensus formation on coevolving networks: groups' formation and
  structure.
\newblock {\em J. Phys. A Math. Gen.}, 2008.

\bibitem{LPW09}
David~Asher Levin, Yuval Peres, and Elizabeth~Lee Wilmer.
\newblock {\em Markov chains and mixing times}.
\newblock Amer. Math. Soc., 2009.

\bibitem{Lig85}
Thomas~M. Liggett.
\newblock {\em Interacting Particle Systems}.
\newblock Springer, New York, USA, 1985.

\bibitem{MM13}
Nishant Malik and Peter~J. Mucha.
\newblock Role of social environment and social clustering in spread of
  opinions in coevolving networks.
\newblock {\em Chaos: An Interdisciplinary Journal of Nonlinear Science},
  23(4), 2013.

\bibitem{N10}
Mark Newman.
\newblock {\em Networks: An Introduction}.
\newblock Oxford University Press, Inc., New York, NY, USA, 2010.

\bibitem{RG13PRE}
Tim Rogers and Thilo Gross.
\newblock Consensus time and conformity in the adaptive voter model.
\newblock {\em Phys. Rev. E}, 88, Sep 2013.

\bibitem{Dur13PRE}
Feng Shi, Peter~J. Mucha, and Richard Durrett.
\newblock Multiopinion coevolving voter model with infinitely many phase
  transitions.
\newblock {\em Phys. Rev. E}, 88, Dec 2013.

\bibitem{SP00}
Brian Skyrms and Robin Pemantle.
\newblock A dynamic model of social network formation.
\newblock {\em Proceedings of the National Academy of Sciences},
  97(16):9340--9346, 2000.

\bibitem{PRL05}
V.~Sood and S.~Redner.
\newblock Voter model on heterogeneous graphs.
\newblock {\em Phys. Rev. Lett.}, 94, May 2005.

\bibitem{PRE05}
Krzysztof Suchecki, V\'ictor~M. Egu\'iluz, and Maxi San~Miguel.
\newblock Voter model dynamics in complex networks: Role of dimensionality,
  disorder, and degree distribution.
\newblock {\em Phys. Rev. E}, 72, Sep 2005.

\bibitem{Vas13}
Federico Vazquez.
\newblock Opinion dynamics on coevolving networks.
\newblock In Animesh Mukherjee, Monojit Choudhury, Fernando Peruani, Niloy
  Ganguly, and Bivas Mitra, editors, {\em Dynamics On and Of Complex Networks,
  Volume 2}, Modeling and Simulation in Science, Engineering and Technology,
  pages 89--107. Springer New York, 2013.

\bibitem{PRL08}
Federico Vazquez, V\'ictor~M. Egu\'iluz, and Maxi~San Miguel.
\newblock Generic absorbing transition in coevolution dynamics.
\newblock {\em Phys. Rev. Lett.}, 100, Mar 2008.

\bibitem{Volz07}
E.~Volz and L.A. Meyers.
\newblock Susceptible-infected recovered epidemics in dynamic contact networks.
\newblock {\em Proc. Biol. Sci.}, 274(1628):2925--2933, 2007.

\bibitem{Volz09}
Erik Volz and Lauren~Ancel Meyers.
\newblock Epidemic thresholds in dynamic contact networks.
\newblock {\em Journal of The Royal Society Interface}, 6(32):233--241, 2009.

\bibitem{PRE13}
Su~Do Yi, Seung~Ki Baek, Chen-Ping Zhu, and Beom~Jun Kim.
\newblock Phase transition in a coevolving network of conformist and contrarian
  voters.
\newblock {\em Phys. Rev. E}, 87, Jan 2013.

\end{thebibliography}
\bibliographystyle{plain}
\end{document}